\documentclass[a4paper,12pt,oneside,reqno]{amsart}

\usepackage[T1]{fontenc}
\usepackage[utf8]{inputenc}
\usepackage{lmodern}
\usepackage[english]{babel}

\usepackage{microtype}

\usepackage[%
	left=2.5cm,       
	right=2.5cm,      
	top=3.5cm,        
	bottom=3.5cm,     
	heightrounded,    
	bindingoffset=0mm 
]{geometry}

\usepackage{amsmath}
\usepackage{amssymb}
\usepackage{mathtools}
\usepackage{mathrsfs}

\usepackage[
colorlinks=true,
urlcolor=teal,
citecolor=magenta,
linkcolor=teal,
breaklinks=true]
{hyperref} 

\usepackage[noabbrev,capitalize,nameinlink]{cleveref}

\usepackage{bookmark}

\usepackage[initials,
msc-links,
]{amsrefs}

\usepackage{graphicx}
\usepackage{xcolor}
\usepackage{enumitem}
\usepackage{url}

\usepackage[final]{showlabels}

\showlabels{bib}

\numberwithin{equation}{section}

\newcommand{\bl}{\mathrm{BL}}
\newcommand{\bv}{\mathrm{BV}}
\newcommand{\ch}{\mathsf{Ch}}
\newcommand{\con}{\mathrm{C}}
\newcommand{\cc}{\mathsf{c}}
\newcommand{\CC}{\mathsf{C}}
\newcommand{\CD}{\mathsf{CD}}
\renewcommand{\d}{\mathsf{d}}
\newcommand{\di}{\,\mathrm{d}}
\newcommand{\dom}{\mathrm{Dom}}
\newcommand{\eps}{\varepsilon}
\newcommand{\heat}{\mathsf{H}}
\newcommand{\leb}{\mathrm{L}}
\newcommand{\lip}{\mathrm{Lip}}
\newcommand{\m}{\mathfrak m}
\newcommand{\meas}{\mathscr{M}}
\newcommand{\N}{\mathbb{N}}
\newcommand{\per}{\mathsf{Per}}
\newcommand{\prob}{\mathscr{P}}
\newcommand{\R}{\mathbb{R}}
\newcommand{\RCD}{\mathsf{RCD}}
\newcommand{\slope}{\mathsf{D}}
\newcommand{\sob}{\mathrm{W}}
\newcommand{\var}{\mathsf{TV}}
\newcommand{\wa}{\mathsf{W}}
\newcommand{\weakto}{\rightharpoonup}
\newcommand{\tcc}{\widetilde\cc}

\renewcommand{\theta}{\vartheta}

\theoremstyle{plain}
\newtheorem{lemma}{Lemma}[section]
\newtheorem{theorem}[lemma]{Theorem}

\newtheorem{corollary}[lemma]{Corollary}

\theoremstyle{definition}
\newtheorem{definition}[lemma]{Definition}
\newtheorem{remark}[lemma]{Remark}

\begin{document}
	
\title[Properties of Lipschitz smoothing heat semigroups]{Properties of Lipschitz smoothing heat semigroups}

\author[N.~De Ponti]{Nicolò De Ponti}
\address[N.~De Ponti]{Universit\`a di Milano-Bicocca, Dipartimento di Matematica e Applicazioni, via Roberto Cozzi 55, 20125 Milano (MI), Italy}
\email{nicolo.deponti@unimib.it}
	
\author[G.~Stefani]{Giorgio Stefani}
\address[G.~Stefani]{Universit\`a degli Studi di Padova, Dipartimento di Matematica ``Tullio Levi-Civita'',
Via Trieste 63, 35121 Padova (PD), Italy}
\email{giorgio.stefani@unipd.it {\normalfont or} giorgio.stefani.math@gmail.com}
	
\date{\today}

\keywords{Heat semigroup, infinitesimal Hilbertianity, smoothing property, indeterminacy, nodal set, Buser inequality, Wasserstein distance}

\subjclass[2020]{Primary 53C23. Secondary 31E05, 58J35}

\thanks{\textit{Acknowledgements}.
The authors thank Lorenzo Dello Schiavo for his valuable comments on a preliminary version of this work, as well as the two referees for their careful reading and observations.
The authors are members of the Istituto Nazionale di Alta Matematica (INdAM), Gruppo Nazionale per l'Analisi Matematica, la Probabilità e le loro Applicazioni (GNAMPA).
This work was initiated and primarily developed while the authors were postdoctoral researchers at the Scuola Internazionale Superiore di Studi Avanzati (SISSA) in Trieste, Italy. The authors would like to express their gratitude to the institution for providing excellent working conditions and a stimulating atmosphere.
}

\begin{abstract}
We prove several functional and geometric inequalities only assuming the linearity and a quantitative $\mathrm{L}^\infty$-to-Lipschitz smoothing of the heat semigroup in metric-measure spaces. 
Our results comprise a Buser inequality, a lower bound on the size of the nodal set of a Laplacian eigenfunction, and different estimates involving the Wasserstein distance.
The approach works in a large variety of settings, including Riemannian manifolds with a variable Kato-type lower bound on the Ricci curvature tensor, $\mathsf{RCD}(K,\infty)$ spaces, and some sub-Riemannian structures, such as Carnot groups, the Grushin plane and the $\mathbb{SU}(2)$ group.
\end{abstract}

\maketitle

\section{Introduction}

\subsection{Framework}

In the last decades, several authors have deeply investigated the connections between fundamental functional and geometric inequalities 
and the properties of the \emph{heat semigroup} $(\heat_t)_{t\ge 0}$, especially 
its \emph{linearity} and \emph{regularizing} nature.
We refer the reader for instance to~\cites{AGS14-C,AGS15,BGL14,EKS15,CJKS20,Gr09,S-C10} and the references therein.

The \emph{linearity} of the heat semigroup is not automatically granted by definition, as for example $(\heat_t)_{t\ge0}$ is not additive in the so-called \emph{Finsler structures}, see~\cite{OS09}.
In the non-smooth framework, the heat semigroup is defined as the $\leb^2$ gradient flow of the \emph{Cheeger energy} (see~\cites{A18} for an account) and its linearity goes under the name of \emph{infinitesimal Hilbertianity} of the ambient space.
This property plays a crucial role in different fundamental aspects of the theory, including the development of 
a powerful non-smooth analogue of Differential Calculus~\cites{Gigli15}.

\emph{Smoothing properties} of 
the heat semigroup, such as the (generalized) \emph{Bakry--\'Emery inequality}~\cites{AGS15,BE85,BGL14,G22},
\begin{equation}
\label{eqi:wbe}
|\nabla\heat_tf|^2
\le 
\kappa(t)^2\,\heat_t\left(|\nabla f|^2\right),
\quad
\text{for}\ t\ge0,
\end{equation}
for a suitable $\kappa\colon[0,\infty)\to(0,\infty)$, usually encode curvature-type information about the ambient space.
For instance, on a complete Riemannian manifold $(M,\mathrm{g})$, the validity of~\eqref{eqi:wbe} with $\kappa(t)=e^{-Kt}$ for some $K\in\R$ is equivalent to the lower bound $\mathrm{Ric}_{\mathrm g}\ge K$ on the Ricci curvature tensor, e.g., see~\cite{vRS05}*{Th.~1.3}.

In passing, we observe that the linearity does not automatically imply any smoothing property of the heat semigroup, see the example in~\cite{AGS14-C}*{Rem.~4.12}.

\subsection{Main aim and results}

An important consequence of~\eqref{eqi:wbe} is the \emph{$\leb^\infty$-to-Lipschitz contraction} of $(\heat_t)_{t\ge0}$ ($\leb^\infty$-to-$\lip$ for short), i.e.,
\begin{equation}
\label{eqi:feller}
f\in \leb^\infty(X)
\implies
\heat_tf\in\lip_b(X)\
\text{with}\
\|\nabla\heat_t f\|_{\leb^\infty}\leq \cc(t)\,\|f\|_{\leb^{\infty}}\
\text{for}\ t>0,
\end{equation}
for a suitable $\cc\colon(0,\infty)\to(0,\infty)$ (see \cref{def:feller} for the precise statement).
Our aim is to show how several functional and geometric inequalities can be deduced uniquely from the linearity of $(\heat_t)_{t\ge0}$ and~\eqref{eqi:feller} in a general  metric-measure space $(X,\d,\m)$ (see \cref{sec:prel} for a detailed description of our setting).

The novelty of our approach lies in its  \textit{minimalistic} point of view, since we do not invoke any stronger curvature-type condition. 
As a byproduct, all results not only come with plain and concise proofs, but also apply to a wide range of examples, including metric-measure spaces with a synthetic constant lower curvature bound, Riemannian manifolds with a variable Kato-type lower bound on the Ricci curvature tensor and several smooth sub-Riemannian structures.
In view of its simplicity and flexibility, we do believe that our strategy may be revisited for other types of semigroups.
We refer to \cref{sec:explicit bound,sec:examples} for the comparison with the  existing literature and the possible  extensions to other settings.

The techniques we employ have been partly applied to some specific frameworks.
However, our work provides new contributions in some contexts in which they were not previously available.
Our main results include but are not limited to:

\begin{enumerate}[label=$\circ$,leftmargin=3ex,itemsep=.5ex,topsep=.5ex]

\item 
\emph{an indeterminacy estimate}: 
a lower bound on the Was\-ser\-ste\-in distance between positive and negative parts of a function $f\in\leb^1\cap\leb^\infty$ in terms of its $\leb^1$ and $\leb^\infty$ norms and the perimeter of its zero set; 

\item
\emph{the size of the nodal set}:
a lower bound on the perimeter of the zero set of a Laplacian eigenfunction $f_{\lambda}$ in terms of its eigenvalue $\lambda$ and of its $\leb^1$ and $\leb^{\infty}$ norms;

\item
\emph{an indeterminacy-type estimate for eigenfunctions}:
a lower bound on the Wasserstein distance between positive and negative parts of an eigenfunction $f_\lambda$ in terms of its eigenvalue $\lambda$ and and of its $\leb^1$ norm;

\item
\emph{a Buser-type inequality}:
an upper bound on the first non-trivial eigenvalue of the Laplacian in terms of the Cheeger constant of the ambient space;

\item 
\emph{a transport--Sobolev inequality}:
an upper bound on the $\leb^1$ norm of a $\bv$ function $f$  in terms of its total variation and of the Wasserstein distance between its positive and negative parts.

\end{enumerate}

The proof of each result 
consists of two main steps.
We first derive \emph{implicit} inequalities depending on $t>0$, and then we provide their \emph{explicit} versions by optimizing with respect to the parameter $t$ in terms of a given upper control on the function $\cc(t)$ in~\eqref{eqi:feller}. 
The precise form of the inequalities depends on the expression of the upper bound on $\cc(t)$---typically, on its asymptotic behavior as $t\to 0^+$.
In all the aforementioned examples, a power-logarithmic-type upper control on $\cc(t)$ is explicitly available.

\subsection{Organization of the paper}
In \cref{sec:prel}, we detail the notation and several preliminary results that we use throughout the paper.
In \cref{sec:implicit bound}, we introduce the $\leb^\infty$-to-$\lip$ property (see \cref{def:feller}) and we deduce its consequences in their implicit form.
In \cref{sec:explicit bound}, by prescribing an  upper bound on  $\cc(t)$ (see~\eqref{eq:controls} and the more general~\eqref{eq:log-controls}), we provide explicit versions of our results.
In \cref{sec:examples}, we discuss the settings to which our approach applies.

\section{Preliminaries}
\label{sec:prel}

\subsection{Function spaces}

We let $(X,\d)$ be a complete and separable metric space.

We let $\con_{b}(X)$ be the space of real-valued, bounded and continuous functions on $X$.
We let $\mathrm{Lip}(X)$, $\mathrm{Lip}_b(X)$ and $\mathrm{Lip}_{bs}(X)$ be the space of Lipschitz functions which are real-valued, bounded and with bounded support, respectively, and we let $\lip(f)\in[0,\infty)$ denote the Lipschitz constant of the function $f\in\lip(X)$. 
 
Given any non-negative Borel measure $\m$ on $X$, for $p\in [1,\infty]$ we let $\leb^p(X,\m)$ be the Lebesgue space of $p$-integrable functions. 
To keep the notation short, we often write $\leb^p(X)$ or simply $\leb^p$ in place of $\leb^p(X,\m)$. 
These spaces will be endowed with the norms
\begin{align*}
\|f\|_{\leb^p}&=\left(\int_X |f|^p\di\m\right)^{\frac{1}{p}} \quad \text{for}\ p\in[1,\infty),
\\[1ex]
\|f\|_{\leb^\infty}&=\inf\big\{C\in[0,\infty)\ :\ |f(x)|\le C \ \text{for $\m$-a.e.} \ x\in X\big\}.
\end{align*}
Note that $\|\cdot\|_{\leb^p}$  is well-defined (possibly equal to $\infty$) on $\m$-measurable functions on~$X$.
As customary, we identify $\leb^p$ functions up to $\m$-negligible sets. 

\subsection{Space of measures}

We let $\meas(X)$ be the space of finite Borel measures on~$X$ and we let $\meas_+(X)=\left\{\mu\in\meas(X):\mu\ge0\right\}$.
We also let 
\begin{equation*}
\mathscr{P}_1(X)
=
\left\{
\mu\in\meas_+(X)
:
\mu(X)=1\
\text{and}\
\int_X\d(x,x_0)\di\mu(x)<\infty\
\text{for some}\ x_0\in X
\right\}
\end{equation*}
be the space of \emph{probability} measures with finite \emph{$1$-moment} on~$X$.

The \emph{total variation} of $\mu\in\meas(X)$ is defined as 
\begin{equation*}
|\mu|(X)
=
\sup\left\{
\int_X f\di\mu
:
f\in \con_b(X),\ \|f\|_{\leb^\infty}\le1
\right\}
\in[0,\infty).
\end{equation*}
Moreover, we say that  $(\mu_k)_{k\in\N}\subset\meas(X)$ \emph{weakly converges} to $\mu\in\meas(X)$, and we write $\mu_k\weakto\mu$ in $\meas(X)$ as $k\to\infty$, if 
\begin{equation}
\label{eq:def_weak_conv}
\lim_{k\to\infty}
\int_Xf\di\mu_k
=
\int_Xf\di\mu
\end{equation}
for every $f\in C_b(X)$.
Finally, we recall that the set
\begin{equation}
\label{eq:deltas}
\mathscr Q(X)
=
\left\{\sum_{k=1}^nc_k\delta_{x_k}: n\in\N,\ x_k\in X, c_k\in\R,\ k=1,\dots,n\right\}
\end{equation}
is dense in $\meas(X)$ with respect to the topology induced by the weak convergence, see~\cite{Bogachev07}*{Ex.~8.1.6(i)} for instance.
As above, we let $\mathscr Q_+(X)=\mathscr Q(X)\cap\mathscr M_+(X)$.

\subsection{Wasserstein \texorpdfstring{$1$}{1}-distance}

The \emph{$1$-Wasserstein distance} $\wa_1$ between $\mu_{1},\mu_{2}\in\meas_+(X)$ is given by   
\begin{equation}
\label{eq:dual_W1}
\wa_1(\mu_1,\mu_2)=\sup\left\{\int_X f\di(\mu_1-\mu_2) \ : \ f\in \lip_b(X), \ \lip(f)\le 1\right\}\in[0,\infty].
\end{equation} 
Whenever $\mu_1,\mu_2\in\meas_+(X)$, a sufficient (but not necessary) condition for $\wa_1(\mu_1,\mu_2)<\infty$ is that $C\mu_1,C\mu_2\in \mathscr{P}_1(X)$ for some constant $C\in (0,\infty)$.
We recall that $(\prob_1(X),\wa_1)$ is a complete and separable metric space.
Moreover, given $(\mu_k)_{k\in\N}\subset\prob_1(X)$ and $\mu\in\prob_1(X)$, $\mu_k\xrightarrow{\wa_1}\mu$ as $k\to\infty$ if and only if $\mu_k\weakto\mu$ in $\meas(X)$ as $k\to\infty$ and, for some $x_0\in X$, 
\begin{equation*}
\lim_{k\to\infty}
\int_X\d(x,x_0)\di\mu_k(x)
=
\int_X\d(x,x_0)\di\mu(x).
\end{equation*}
Finally, thanks to~\cite{V09}*{Th.~6.18}, the set $\mathscr Q_+(X)\cap\prob(X)$ is  $\wa_1$-dense in $\prob_1(X)$.

\subsection{\texorpdfstring{$\bl^\star$}{BL*} distance}

The \emph{bounded-Lipschitz dual distance} between $\mu_1,\mu_2\in\meas(X)$ is defined as
\begin{equation*}
\bl^\star(\mu_1,\mu_2)
=
\sup
\left\{
\int_Xf\di(\mu_1-\mu_2)
:
f\in\lip_b(X),\ \|f\|_{\bl}\le1
\right\}\in[0,\infty),
\end{equation*}
where $\|f\|_{\bl}=\max\{\lip(f),\,\|f\|_{\leb^\infty}\}$ for every $f\in\lip_b(X)$.
We recall the following result, see~\cite{Bogachev07}*{Th.~8.3.2} for instance.

\begin{theorem}
\label{res:bl_top}
The topology induced by the $\bl^\star$  distance coincides with the topology induced by the weak convergence~\eqref{eq:def_weak_conv} on $\meas_+(X)$.
\end{theorem}

\subsection{AC measures}
From now on, we assume that the \emph{reference measure}~$\m$ is a non-negative Borel-regular measure which is finite
on bounded sets and such that $\operatorname{supp}\m=X$.

We define the set of \emph{absolutely continuous measures} on $X$ as 
\begin{equation*}
\meas^{\rm ac}(X)=\{\mu\in\meas(X) : \mu\ll\m\}
\end{equation*}
and we set $\meas^{\rm ac}_+(X)=\{\mu\in\meas^{\rm ac}(X):\mu\ge0\}$. 
Clearly, $\mu\in\meas^{\rm ac}(X)$ if and only if $\mu=f\m$ for some $f\in L^1(X)$ and, moreover, $\mu\in\meas^{\rm ac}_+(X)$ if and only if $f\ge0$.

Thanks to \cref{res:bl_top}, we have the following approximation result, whose proof is briefly sketched below for the convenience of the reader.

\begin{corollary}
\label{res:density_ac}
The set $\meas^{\rm ac}_+(X)$ is dense in $\meas_+(X)$ with respect to the $\bl^\star$ distance.
\end{corollary}

\begin{proof}
Since the set $\mathscr Q(X)$ defined in~\eqref{eq:deltas} is dense in $\meas(X)$ with respect to the topology induced by the weak convergence (recall~\cite{Bogachev07}*{Ex.~8.1.6(i)}), by \cref{res:bl_top} it is not restrictive to assume that $\mu\in\mathscr Q_+(X)$.
By the triangular inequality and by homogeneity of the $\bl^\star$ distance, it is not restrictive to further assume that $\mu=\delta_{\bar x}$ for some $\bar x\in X$.
In this case, owing to the fact that $\m$ is finite
on bounded sets and $\operatorname{supp}\m=X$, we can define $\mu_\eps=f_\eps\m$ with $f_\eps
=
\frac{\chi_{B_\eps(\bar x)}}{\m(B_\eps(\bar x))}$ for $\eps>0$.
Observing that 
\begin{equation*}
\left|
\int_Xg\di(\delta_{\bar x}-\mu_\eps)
\right|
=
\left|
\int_{B_\eps(\bar x)}\frac{g-g(\bar x)}{\m(B_\eps(\bar x))}\di \m
\right|
\le 
\lip(g)
\int_{B_\eps(\bar x)}\frac{\d(x,\bar x)}{\m(B_\eps(\bar x))}\di \m(x)
<
\lip(g)
\,\eps
\end{equation*}
whenever $g\in\lip_b(X)$, we get that $\bl^\star(\delta_{\bar x},\mu_\eps)\le\eps$, concluding the proof.
\end{proof}

\subsection{Slope}

The \emph{slope} of $f\in\lip(X)$ is defined as
\begin{equation*}
|\slope f|(x)=
\begin{cases}
\displaystyle
\limsup_{y\to x}\frac{|f(y)-f(x)|}{\d(y,x)}
&
\text{if $x\in X$ is an accumulation point},
\\[1ex]
0
&
\text{if $x\in X$ is isolated}.
\end{cases}
\end{equation*}

\subsection{Relaxed gradient}

Since the set 
$\left\{f\in \leb^2(X) : f\in\mathrm{Lip}_b(X),\ |\slope f|\in \leb^2(X)\right\}$
is dense in~$\leb^2(X)$, we can say that $G\in \leb^2(X)$ is a \emph{relaxed gradient} of $f\in \leb^2(X)$ if there exists a sequence $(f_k)_{k\in\N}\subset \leb^2(X)\cap\lip(X)$ such that
$f_k\to f$ in $\leb^2(X)$ and $|\slope f_k|\rightharpoonup \widetilde{G}$ in $\leb^2(X)$ for some $\widetilde G\in\leb^2(X)$ such that $\widetilde{G}\le G$ $\m$-a.e.\ in $X$.

The set of all the relaxed gradients of $f\in\leb^2(X)$ is a closed and convex subset of $\leb^2(X)$. 
Thus, when such set is not empty, it admits an element of minimal $\leb^2$ norm, called the \emph{minimal relaxed gradient} and denoted by $|\slope f|_{w}$. 
Such element is minimal also in the $\m$-a.e.\ sense, meaning that $|\slope f|_{w}\le G$ $\m$-a.e.\ for any relaxed gradient $G$ of $f$.
In particular, $|\slope f|_{w}\le |\slope f|$ $\m$-a.e.\ for every $f\in \lip_{bs}(X)$.

\subsection{Cheeger energy}

We let 
\begin{equation*}
\ch(f)
=
\inf\left\{
\liminf_{k\to\infty}
\frac12\int_X|\slope f_k|^2\di\m
:
f_k\in\lip_{bs}(X,\m),
\
f_k\to f\ \text{in}\ \leb^2(X,\m)\ \text{as}\ k\to\infty
\right\}
\end{equation*}
be the \emph{Cheeger energy} of $f\in \leb^2(X)$.
Thanks to~\cite{AGS14-C}*{Ths.~6.2 and 6.3}, we can write
\begin{equation*}
\ch(f)=
\begin{cases}
\displaystyle\frac{1}{2}\int_X |\slope f|^2_{w}\di\m
&
\text{if $f$ admits a relaxed gradient},
\\[2ex]
+\infty
&
\text{otherwise}.
\end{cases}
\end{equation*} 
As usual, we set
\begin{equation*}
\sob^{1,2}(X)
=
\sob^{1,2}(X,\d,\m)
=
\left\{
f\in \leb^2(X)
:
\ch(f)<\infty
\right\}.
\end{equation*}
The Cheeger energy is a $2$-homogenous, convex and lower semicontinuous functional on $\leb^2(X)$ and the set $\sob^{1,2}(X)$, endowed with the norm
\begin{equation*}
\|f\|^2_{\sob^{1,2}}=\|f\|^2_{\leb^2}+2\ch(f),
\quad
f\in\sob^{1,2}(X),
\end{equation*}
is dense in~$\leb^2(X)$.

\subsection{Laplacian operator}

We let $\partial^{-}\ch(f)\subset \leb^2(X)$ be the \emph{subdifferential} of $\ch$ at $f\in \leb^2(X)$, i.e., $\ell\in\partial^{-}\ch(f)$ if and only if 
\begin{equation*}
\ch(g)\ge
\ch(f)
+
\int_X \ell\,(g-f)\di\m  \quad \textrm{for all}\ g\in \leb^2(X).
\end{equation*}
We write $f\in\dom(\Delta)$ if $f\in \leb^2(X)$ is such that $\partial^{-}\ch(f)\neq \emptyset$.
For $f\in \dom(\Delta)$, we let $\Delta f$ be the element of minimal $\leb^2$ norm in $-\partial^{-}\ch(f)$ and we call it the \emph{Laplacian} of~$f$. 

\subsection{Heat semigroup}

By the classical theory of gradient flow in Hilbert spaces, for every $f\in \leb^2(X)$ there exists a unique locally Lipschitz curve $t\mapsto \heat_tf$ from $(0,\infty)$ to $\leb^2(X)$, called the \emph{heat flow at time~$t$ starting from~$f$}, such that 
\begin{equation}
\label{eq:heat_flow}
\begin{cases}\dfrac{\mathrm{d}}{\mathrm{d} t}\heat_tf=\Delta \heat_tf
&
\text{for a.e.}\ t \in (0,\infty),
\\[1.5ex]
\heat_tf\to f\ \text{in}\ \leb^2(X)
&
\text{as}\ t\to 0^{+}.
\end{cases}
\end{equation}
We let $\heat_0=\mathrm{Id}$ be the identity operator in $\leb^2(X)$, so that $(\heat_t)_{t\ge 0}$ is a (possibly, non-linear) semigroup, called the \emph{heat semigroup}. 
Because of the $2$-homogeneity of the Cheeger energy, $\heat_t$ and $\Delta$ are $1$-homogeneous, i.e.,
\begin{equation}
\label{eq:1-homo}
\begin{array}{lr}
\heat_t(\lambda f)=\lambda\, \heat_tf
&
\text{for}\
f\in \leb^2(X),\ 
\lambda\in\R,
\\[1ex]
\Delta(\lambda g)=\lambda\,\Delta g
&
\text{for}\
\
g\in \dom(\Delta),\ 
\lambda\in\R.
\end{array}
\end{equation}
Moreover, for $t\ge0$ and $p\in[1,\infty]$, the heat semigroup satisfies the \emph{contraction property}
\begin{equation}
\label{eq:contraction}
\|\heat_tf- \heat_tg\|_{\leb^p}\le \|f-g\|_{\leb^p} \quad \forall f,g\in \leb^2(X)\cap \leb^p(X)
\end{equation}
and the \emph{maximum principle}
\begin{equation}
\label{eq:maxprinciple}
f\leq C\
\text{$\m$-a.e.\ in $X$ for some $C\in\R$}
\implies
\heat_tf\leq C\
\text{$\m$-a.e.\ in $X$}
\end{equation}
(in particular, $\heat_t$ is \emph{sign preserving}). 
Finally, assuming that
\begin{equation}
\label{ass: meas exp}
\exists\,
A,B>0\ \text{and}\ \bar x\in X\ \text{such that}\ 
\m(\{x\in X:\d(x,\bar{x})< r\})\le Ae^{Br^2}\
\text{for all}\ r>0,
\end{equation}
the heat semigroup also satisfies the \emph{mass preserving property}
\begin{equation}
\label{eq:mass_preserving}
\int_{X}\heat_tf\di\m=\int_{X}f\di\m 
\quad \text{whenever $f \in \leb^2(X)\cap \leb^1(X)$ and $t>0$}.
\end{equation}

\subsection{Infinitesimal Hilbertianity and non-smooth Calculus}

From now on, we assume that the metric-measure space $(X,\d,\m)$ is \emph{infinitesimally Hilbertian}, meaning that
\begin{equation}
\label{eq:inf_hilb}
2\ch(f)+2\ch(g)=\ch(f+g)+\ch(f-g)
\quad
\text{for all}\
f,g\in \sob^{1,2}(X)
.
\end{equation}
In this case, the heat flow is also additive, and thus $(\heat_t)_{t\ge 0}$ is a linear semigroup with the energy $\mathcal{E}=2\ch$ being the associated \emph{strongly-local Dirichlet form}.

By the density of $\leb^{2}(X)\cap \leb^{p}(X)$ in $\leb^{p}(X)$ and in virtue of~\eqref{eq:contraction}, $\heat_t$ extends to a strongly continuous linear semigroup of contractions in $\leb^{p}(X)$ for any $p\in [1,\infty)$, for which we keep the same notation. 
By duality, $\heat_t$ also extends to a linear and weakly$^\star $-continuous semigroup of contractions in $\leb^{\infty}(X)$ such that
\begin{equation}
\label{eq:heat_in_L_infty}
\int_X g\,\heat_tf\di\m
=
\int_X f\,\heat_tg\di\m 
\quad 
\text{for}\ f\in \leb^{\infty}(X)\ \text{and}\ g\in \leb^1(X).
\end{equation} 

By polarization, there exists a  bilinear form 
\begin{equation*}
(f,g)\mapsto
\int \slope f\cdot \slope g \di\m
\le
\int_X |\slope f|_{w}\,|\slope g|_{w}\di\m 
\quad
\text{for}\
f,g\in \sob^{1,2}(X)
\end{equation*}
satisfying the \emph{integration-by-parts}
\begin{equation}
\label{eq:int grad-grad}
\int_X \slope f\cdot \slope g \, \mathrm d\m 
=
-\int_X f\,\Delta g \, \mathrm d\m
\quad
\text{for}\
f\in \sob^{1,2}(X)\text{ and }g\in \dom(\Delta)\, .
\end{equation}
In addition, the heat semigroup and the Laplacian are \emph{self-adjoint}, i.e.,
\begin{align}
\label{eq: Lapl self}
\int_X f\,\Delta g \di\m 
&=
\int_X g\,\Delta f \di\m \quad 
\text{for} \ f,g\in \dom(\Delta),
\\[1ex]
\label{eq: heat self}
\int_X f\,\heat_t g \di\m
&= 
\int_X g\,\heat_t f \di\m
\quad 
\text{for}\ f,g\in \leb^2(X)\ \text{and} \ t\ge0.
\end{align}
Finally, we recall the commutation
\begin{equation}
\label{eq: Lapl-heat comm}
\heat_t(\Delta f)
=
\Delta \heat_tf
\quad 
\text{for}\ 
f\in \dom(\Delta)\ \text{and} \ t>0
\end{equation} 
and the \textit{a priori} estimate 
\begin{equation}
\label{eq: Lapl heat a priori}
\|\Delta \heat_tf\|_{\leb^2}\le \frac{1}{t}\,\|f\|_{\leb^2} \quad 
\text{for}\ f\in \leb^2(X)\ \text{and} \ t>0.
\end{equation}

Using \eqref{eq: Lapl self},
\eqref{eq: heat self} 
and 
\eqref{eq: Lapl-heat comm}, together with the fact that the heat flow is a semigroup with image contained in the domain of the Laplacian, we get that 
\begin{equation}
\label{eq:Lapl-heat_double_self}
\int_X g\,\Delta \heat_t f\di\m
= 
\int_X f\, \Delta \heat_t g\di\m 
\quad 
\text{for} \ f,g \in \leb^2(X)\ \text{and}\ t>0.
\end{equation}

The following result is a simple consequence of the above properties. 
Although this result may be known to experts, we give its short proof here for the reader's convenience.

\begin{lemma}
\label{res:heat_ibp}
If $f,g\in \sob^{1,2}(X)$, then 
\begin{equation*}
\int_X g\,(f-\heat_tf)\di\m
=
\int_0^t\int_X \slope f\cdot \slope\heat_s g\di\m\di s
\quad
\text{for all}\ t>0.
\end{equation*}
\end{lemma}

\begin{proof}
Since $\ch$ is quadratic and $f\in \sob^{1,2}(X)$, we know that 
\begin{equation*}
s\mapsto \heat_sf\in C^1((0,+\infty);\dom(\Delta))
\cap 
C^0([0,+\infty);\leb^2(X))
\end{equation*}
with 
\begin{equation*}
\lim_{h\to0}
\frac{\heat_{s+h}f-\heat_sf}h=\Delta \heat_sf
\quad
\text{in}\ \leb^2(X)\ 
\text{for}\ s>0.
\end{equation*}
As a consequence, thanks to~\eqref{eq:Lapl-heat_double_self}, we can compute
\begin{equation*}
\frac{\mathrm d}{\mathrm ds}\int_X g\,\heat_sf\di\m
=
\int_Xg\,\Delta \heat_sf\di\m
=
\int_Xf\,\Delta \heat_sg\di\m
\quad
\text{for}\ s\in(0,t).
\end{equation*}
By~\eqref{eq:int grad-grad}, we can integrate by parts to obtain 
\begin{equation*}
\int_X f\,\Delta \heat_sg\di\m
=
-\int_X \slope f\cdot \slope\heat_sg\di\m
\quad
\text{for}\ s\in(0,t).
\end{equation*}
We can hence integrate in $s\in(0,t)$ to get 
\begin{equation*}
\int_X g\,(f-\heat_tf)\di\m
=
-
\int_0^t
\frac{\mathrm d}{\mathrm ds}\int_X g\,\heat_sf\di\m
\di s
=
\int_0^t
\int_X \slope f\cdot \slope \heat_sg\di\m
\di s,
\end{equation*}
where the right-hand side is well defined since $g\in \sob^{1,2}(X)$ (see~\cite{G22}*{Lem.~2.1}). 
\end{proof}

\subsection{Eigenfunctions and spectrum}

A non-zero $f_{\lambda}\in \dom(\Delta)$ is an \emph{eigenfunction} of the Laplacian relative to the \emph{eigenvalue} $\lambda \in [0,\infty)$ (\emph{$\lambda$-eigenfunction}, for short) if 
\begin{equation*}
-\Delta f_{\lambda}=\lambda\, f_{\lambda}.
\end{equation*}

If $\m(X)<\infty$, then any non-zero constant function is a $0$-eigenfunction and, moreover, every other $\lambda$-eigenfunction~$f_\lambda$  has zero mean, so that
\begin{equation}
\label{eq:eigen_zero_mean}
\int_{X}f_\lambda^-\di\m=\int_{X}f_\lambda^+\di\m.
\end{equation}

For the reader's ease, we recall the following well-known result.

\begin{lemma}
\label{res:heated_eigenf}
If $f_{\lambda}$ is a $\lambda$-eigenfunction, then $\heat_tf_\lambda=e^{-\lambda t}f_\lambda$ for all $t\ge0$.
\end{lemma}



The \emph{Rayleigh quotient} of $f\in \dom(\ch)\setminus\{0\}$ is defined as
\begin{equation}\label{eq: Rayleigh} 
\mathcal{R}(f)=\frac{2\ch(f)}{\displaystyle\int_X |f|^2\di\m}.
\end{equation}
We consider the variational quantities
\begin{equation}
\label{eq:eig0}
\lambda_0
=
\inf
\left\{
\mathcal R(f)
:
f\in\dom(\ch)\setminus\{0\}
\right\}
\end{equation}
and 
\begin{equation}
\label{eq:eig1}
\lambda_1
=
\inf
\left\{
\mathcal R(f)
:
f\in\dom(\ch)\setminus\{0\},\ 
\int_X f\di\m=0
\right\}.    
\end{equation}
Clearly, $0\le\lambda_0\le \lambda_1$.
Moreover, if $\lambda_k$, $k=0,1$, is smaller than the infimum of the essential spectrum of $-\Delta\colon\dom(\Delta)\to\leb^2(X)$, then $\lambda_k$ corresponds to the classical $k$-th eigenvalue of $-\Delta$ in the spectral sense, see~\cite{D95}*{Th.~4.5.2}.

For the convenience of the reader, we recall the following simple result.

\begin{lemma}
\label{res:heat_lambdas_10}
Let $f\in \leb^2(X)$.
The following hold:

\begin{enumerate}[label=(\roman*),itemsep=1ex]

\item
\label{item:heat_lambda_1}
if $\m(X)<\infty$ and $\displaystyle\int_X f\di\m=0$, then $\|\heat_tf\|_{\leb^2}\le e^{-\lambda_1t}\|f\|_{\leb^2}$ for all $t\ge0$;

\item
\label{item:heat_lambda_0}
if $\m(X)=\infty$, then $\|\heat_tf\|_{\leb^2}\le e^{-\lambda_0t}\|f\|_{\leb^2}$ for all $t\ge0$.

\end{enumerate}
\end{lemma}

\begin{proof}
To prove~\ref{item:heat_lambda_1}, we can assume $\m(X)=1$. 
By definition of $\lambda_1$ in~\eqref{eq:eig1}, we have
\begin{equation*}
2\lambda_1\int_X |\heat_tf|^2\di\m 
\leq 
2\int_X |\slope (\heat_tf)|_{w}^2\di\m
=
-2\int_X \heat_tf\Delta(\heat_tf)\di\m
=
-\frac{\mathrm d}{\mathrm dt}\int_X |\heat_tf|^2\di\m
\end{equation*}
for all $t>0$, thanks to~\eqref{eq:int grad-grad}, \eqref{eq:mass_preserving} and~\eqref{eq:heat_flow}.
Hence~\ref{item:heat_lambda_1} follows by Gr\"onwall's Lemma.
The proof of~\ref{item:heat_lambda_0} similarly follows by exploiting the definition in~\eqref{eq:eig0}  and is thus omitted.
\end{proof}

\subsection{\texorpdfstring{$\bv$}{BV} functions and Cheeger constants} 

We say that $f\in \bv(X)=\bv(X,\d,\m)$ if $f\in\leb^1(X)$ and there exists $(f_k)_{k\in\N}\subset\lip_{bs}(X)$ such that $f_k\to f$ in $\leb^1(X)$ and 
\begin{equation*}
\sup_{k\in\N} 
\int_X |\slope f_k|\di\m  
<\infty.
\end{equation*}
We thus let 
\begin{equation}
\label{eq:var}
\var(f)
=
\inf\left\{\liminf_{k\to \infty}\int_X |\slope f_k|\di\m: f_k\in \lip_{bs}(X),\ f_k\to f \ \mathrm{in} \ \leb^1(X)\ \text{as}\ k\to\infty\right\}
\end{equation}
be the \emph{total variation of $f$}.
We write $\per(A)=\var(\chi_A)$ whenever $\chi_A\in\bv(X)$.

In analogy with~\eqref{eq:eig0} and~\eqref{eq:eig1}, we consider
\begin{equation}
\label{eq:Cheeger_0}
h_0(X)=\inf \left\{\frac{\per(A)}{\m(A)} :  A\subset X\ \text{Borel subset with}\ 0<\m(A)<\infty\right\}
\end{equation}  
and
\begin{equation}
\label{eq:Cheeger_1}
h_1(X)=\inf  \left\{\frac{\per(A)}{\m(A)} : A\subset X\ \text{Borel subset with}\ 0<\m(A)\leq \frac{\m(X)}2 \right\}. 
\end{equation}
The definition in~\eqref{eq:Cheeger_1} corresponds to the one introduced in~\cite{C69}.
We observe that, if $\m(X)<\infty$, then $h_0(X)=0$.

For future convenience, we recall the following simple estimate, proved in~\cite{DF22}*{Lem.~3.2}. 

\begin{lemma}
\label{res:norm-cheeg}
If $\m(X)<\infty$, then
\begin{equation}
\label{eq:norm-cheeg}
h_1(X)
\le
\inf\left\{ 
2\,\per(\lbrace f > 0\rbrace)\,\frac{\,\|f\|_{\leb^{\infty}}}{\|f\|_{\leb^1}}
:
f\in \leb^{\infty}(X)\ \text{such that}\ \int_X f\di\m=0
\right\}.
\end{equation}
\end{lemma}

\section{Quantitative Lipschitz smoothing property and implicit inequalities}

In this section, we study the consequences of the quantitative Lipschitz smoothing property (see \cref{def:feller} below) in their implicit form.

\label{sec:implicit bound}

\subsection{Main assumptions}
\label{subsec:ass}

From now on, we let $(X,\d,\m)$ be a metric-measure space satisfying the following properties:
\begin{enumerate}[label=(P\arabic*),ref=P\arabic*,itemsep=.5ex,leftmargin=6ex,topsep=.5ex]

\item\label{item:ass_d}
$(X,\d)$ is a complete and separable metric space;

\item\label{item:ass_m} 
$\m$ is a non-negative Borel-regular measure on $X$ satisfying $\operatorname{supp}\m=X$ and~\eqref{ass: meas exp};

\item\label{item:ass_ch}
$(X,\d,\m)$ is infinitesimally Hilbertian, i.e., \eqref{eq:inf_hilb} holds.

\end{enumerate}

Besides the above properties, we assume the validity of the following property, which is the main assumption of our paper.
Here and in the following, we let $\cc\colon(0,\infty)\to(0,\infty)$ be a Borel function.

\begin{definition}[$\leb^\infty$-to-$\lip$]
\label{def:feller}
We say that the heat semigroup $(\heat_t)_{t\ge 0}$ satisfies the \emph{$\leb^\infty$-to-$\lip$ contraction property with Lipschitz constant~$\cc$} (or is~\ref{eq:c-Feller}, for short)
if
\begin{equation}
\label{eq:c-Feller}
\tag{$\cc$-$\lip$}
f\in \leb^\infty(X)
\implies
\heat_tf\in\lip_b(X)\
\text{with}\
\lip(\heat_t f)\leq \cc(t)\,\|f\|_{\leb^{\infty}}\
\text{for all}\ 
t>0.
\end{equation}
\end{definition}

Under~\eqref{eq:c-Feller}, we can define the \emph{optimal} function $\cc^\star\colon(0,\infty)\to(0,\infty)$ by letting
\begin{equation}
\label{eq:c-star}
\cc^\star(t)
=
\sup\left\{\lip(\heat_t f) : f\in\leb^\infty(X)\ \text{with}\ \|f\|_{\leb^\infty}=1\right\}
\end{equation}
for all $t>0$.
We have the following result, which motivates the definition of $\cc^\star$ in~\eqref{eq:c-star}.

\begin{lemma}[Properties of $\cc^\star$]
\label{res:c-star}
Assume that~\eqref{eq:c-Feller} holds for some Borel $\cc\colon(0,\infty)\to(0,\infty)$ and let $\cc^\star\colon(0,\infty)\to(0,\infty)$ be as in~\eqref{eq:c-star}.
Then, the following hold:
\begin{enumerate}[label=(\roman*),itemsep=1ex]
\item
\label{item:c-star_feller}
$\cc^\star$  satisfies~\eqref{eq:c-Feller};
\item
\label{item:c-star_minimal}
$\displaystyle\cc^\star(t)\le\inf\{\cc(s): s\in(0,t]\}$ for all $t>0$;
\item 
\label{item:c-star_decres}
$\cc^\star$ is non-increasing, $\cc^\star(t)\le\cc^\star(s)$ for all $0<s\le t$;
\item
\label{item:c-star_int}
if $\cc\in\leb^1((0,t_0])$ for some $t_0>0$, then $\cc^\star(t)=O(1/t)$ as $t\to0^+$.
\end{enumerate} 
\end{lemma}

\begin{proof}
Since~\eqref{eq:c-Feller} holds for $\cc$, this means that $\heat_tf\in\lip_b(X)$ for every $t>0$ and $f\in\leb^\infty(X)$.
By the definition in~\eqref{eq:c-star} and~\eqref{eq:1-homo}, we also have that 
\begin{equation*}
\lip(\heat_tf)\le\cc^\star(t)\|f\|_{\leb^\infty}
\end{equation*}
for every $t>0$ and $f\in\leb^\infty(X)$, proving~\ref{item:c-star_feller}.
Moreover, as a consequence of the semigroup property and the maximum principle, we have
\begin{equation*}
\lip(\heat_tf)
\le 
\lip(\heat_s \heat_{t-s}f)
\le
\cc(s) 
\,
\|\heat_{t-s} f\|_{\leb^\infty}
\le
\cc(s) 
\end{equation*}
whenever $0<s\le t$ and $f \in \leb^\infty(X)$ is such that $\|f\|_{\leb^\infty}
=1$. 
By~\eqref{eq:c-star}, we hence get that 
\begin{equation*}
\cc^\star(t)
=
\sup\left\{\lip(\heat_t f) : f\in\leb^\infty(X)\ \text{with}\ \|f\|_{\leb^\infty}=1\right\}
\le
\cc(s)
\end{equation*}
for $0<s\le t$, proving~\ref{item:c-star_minimal}.
Similarly, in view of~\ref{item:c-star_feller},  we can also write
\begin{equation*}
\lip(\heat_tf)
\le 
\lip(\heat_s \heat_{t-s}f)
\le
\cc^\star(s) 
\,
\|\heat_{t-s} f\|_{\leb^\infty}
\le
\cc^\star(s)
\end{equation*}
whenever $0<s\le t$ and $f \in \leb^\infty(X)$ is such that $\|f\|_{\leb^\infty}
=1$, so that~\ref{item:c-star_decres} follows from~\eqref{eq:c-star}.
Finally, if $\cc\in\leb^1((0,t_0])$ for some $t_0>0$, then also $\cc^\star\in\leb^1((0,t_0])$ for $t_0>0$ by~\ref{item:c-star_minimal}.
Therefore, thanks to~\ref{item:c-star_decres}, we can estimate
\begin{equation*}
\int_0^{t_0}\cc^\star(s)\di s
\ge
\int_0^t\cc^\star(s)\di s
\ge 
\int_0^t
\cc^\star(t)
\di s
=
t\,\cc^\star(t)
\end{equation*}
for every $t\in(0,t_0]$, proving~\ref{item:c-star_int} and concluding the proof.
\end{proof}

By \cref{res:c-star}, it is not restrictive to assume that~\eqref{eq:c-Feller} holds for some non-increasing function $\cc\colon(0,\infty)\to(0,\infty)$. 
Moreover, by additionally requiring that $\cc\in\leb^1_{\rm loc}([0,\infty))$, it is not restrictive to also assume that $\cc(t)=O(1/t)$ as $t\to0^+$.
However, apart from the additional assumption that $\cc\in\leb^1_{\rm loc}([0,\infty))$ (see~\eqref{eq:cc_int_loc} below), we do not need these further properties of $\cc$ in our results.

\subsection{Comparison with \texorpdfstring{$\lip$-to-$\lip$}{lip-to-lip} contraction}

Before going further, it is worth comparing~\eqref{eq:c-Feller} with the following different condition, see~\cite{AGS15}*{Sec.~3.2}.

\begin{definition}[$\lip$-to-$\lip$]
We say that the heat semigroup $(\heat_t)_{t\ge0}$ satisfies the \emph{$\lip$-to-$\lip$ contraction property with constant $\cc$} if 
\begin{equation}
\label{eq:c-lip-to-lip}
f\in\leb^2(X)\cap\lip_b(X)
\implies
\heat_tf\in\lip_b(X)\
\text{with}\
\lip(\heat_tf)
\le\cc(t)\,\lip(f)
\
\text{for all}\ t>0.
\end{equation}
\end{definition} 

We have the following result, showing that~\eqref{eq:c-Feller} is actually \emph{stronger} than~\eqref{eq:c-lip-to-lip} in metric-measure spaces with finite diameter.

\begin{lemma}
\label{res:diam_ltl}
Let $d=\operatorname{diam}(X)<\infty$.
If~\eqref{eq:c-Feller} holds with constant $\cc$, then also~\eqref{eq:c-lip-to-lip} holds with constant $d\,\cc$.
\end{lemma}

\begin{proof}
If $f\in\leb^2(X)\cap\lip_b(X)$ and $x_0\in X$, then $g=f-f(x_0)\in\lip_b(X)$.
By~\eqref{eq:c-Feller}, we hence get that $\heat_tg\in\lip_b(X)$ with $\lip(\heat_tg)\le \cc(t)\|g\|_{\leb^\infty}$ for all $t>0$.
However, it is easily seen that $\lip(\heat_tg)=\lip(\heat_tf)$ for all $t>0$ and $\|g\|_{\leb^\infty}=\|f-f(x_0)\|_{\leb^\infty}\le\lip(f)\operatorname{diam}(X)=d\,\lip(f)$, immediately yielding the conclusion.
\end{proof}

In view of \cref{res:diam_ltl},
it would be interesting to investigate the existence of a metric-measure space $(X,\d,\m)$ with $\operatorname{diam}(X)=\infty$ such that~\eqref{eq:c-Feller} holds but~\eqref{eq:c-lip-to-lip} does not.

\subsection{Dual semigroup}

For any $f\in\leb^1(X)$ and $t\ge0$, we define
\begin{equation}
\label{eq:heatast}
\heat^\star_t(f\m)=(\heat_tf)\,\m
\in\meas(X).
\end{equation}
In the following result, we show that the semigroup $\heat_t^\star$ in~\eqref{eq:heatast} can be extended to finite Borel measures on $X$.
Here and in the following, we let $\mathrm{rba}(X)$ be the space of bounded, Borel regular and finitely-additive measures on $X$ (for more details, see e.g.~\cite{AB06}*{Ch.~14}). 

\begin{theorem}
\label{res:dual_heat}
If $\mu\in\meas(X)$ and  $t>0$, then there exists a unique $\heat_t^\star\mu\in\mathrm{rba}(X)$ with $|\heat_t^\star \mu|(X)\le|\mu|(X)$ such that 
\begin{equation}
\label{eq:dual_heat}
\int_X \heat_tf\di\mu
=
\int_Xf\di \heat_t^\star\mu
\quad
\text{for all}\
f\in \con_b(X).
\end{equation}
If $(X,\d)$ is locally compact, then $\heat_t^\star\mu\in\meas^{\rm ac}(X)$, with $\heat_t^\star\mu\ge0$ if $\mu\ge0$. 
\end{theorem}

\begin{proof}
Let $t>0$ be fixed. 
Thanks to~\eqref{eq:c-Feller}, the map $\mathcal F\colon\con_b(X)\to\R$ given by
\begin{equation*}
\mathcal F(f)
=
\int_X \heat_tf\di\mu,
\quad
\text{for}\ f\in \con_b(X),
\end{equation*}
defines a linear and continuous functional on $\con_b(X)$ such that, by~\eqref{eq:maxprinciple}, 
\begin{equation*}
|\mathcal F(f)|
\le 
\|\heat_tf\|_{\leb^\infty}\,|\mu|(X)
\le 
\|f\|_{\leb^\infty}\,|\mu|(X)
\quad
\text{for}\ f\in \con_b(X).
\end{equation*} 
By~\cite{AB06}*{Th.~14.10}, we hence get that $\heat_t^\star\mu\in\mathrm{rba}(X)$.
\textit{A fortiori}, the restriction of $\mathcal F$ to $\con_c(X)$, the space of continuous functions with compact support, is a linear continuous operator on $\con_c(X)$.
If $(X,\d)$ is locally compact, then, by~\cite{AB06}*{Ths.~14.12 and~14.14}, we get that $\heat_t\mu\in\meas(X)$ with $\heat_t\mu\ge0$ if $\mu\ge0$.
To conclude, we just need to prove that $|\heat_t\mu|\ll\m$.
Let $K\subset X$ be a compact set such that $\m(K)=0$. 
We can find a sequence $(f_k)_{k\in\N}\subset\con_c(X)$, $f_k(x)=[1-k\,\d(x,K)]^+$, $x\in X$, such that $\chi_K\le f_k\le\chi_H$ for $k\in\N$, where $H=\{x\in X:\d(x,K)\le1\}$, and $f_k(x)\to\chi_K(x)$ for all $x\in X$ as $k\to\infty$. 
Since also $f_k\to\chi_K$ in $\leb^2(X)$ as $k\to\infty$, we can apply the Dominated Convergence Theorem twice to infer that
\begin{equation*}
\heat_t^\star\mu(K)
=
\int_K\di\heat_t^\star\mu
=
\lim_{k\to\infty}
\int_X f_k \di\heat_t^\star\mu
=
\lim_{k\to\infty}
\int_X \heat_t f_k \di\mu
=
\int_X \heat_t \chi_K \di\mu
=0.
\end{equation*}
By inner regularity, we thus get $\heat_t^\star\mu(A)=0$ on any Borel set $A\subset X$ with $\m(A)=0$.
\end{proof}

As suggested by the anonymous referee, \cref{res:dual_heat} can be improved for non-negative measures avoiding the assumption on the local compactness as stated in \cref{res:referee_bl} below.
The proof of this result is based on the (dual formulation) of the inequality
\begin{equation}
\label{eq:ltbl}
\|\heat_tf\|_{\bl}
\le 
\tcc(t)\,\|f\|_{\leb^\infty},
\end{equation} 
valid for all $f\in\leb^\infty(X)$ and $t>0$, where $\tcc(t)=\max\{\cc(t),1\}$, which is an immediate consequence of~\eqref{eq:c-Feller}.

\begin{theorem}
\label{res:referee_bl}
If $f_0,f_1\in\leb^1(X)$ with $f_0,f_1\ge0$, then 
\begin{equation}
\label{eq:referee_bl_ac}
\|\heat_t(f_0-f_1)\|_{\leb^1}
\le 
\tcc(t)
\,
\bl^\star(f_0\m,f_1\m)
\quad
\text{for all}\ t>0.
\end{equation}
As a consequence, if $\mu_0,\mu_1\in\meas_+(X)$, then
\begin{equation}
\label{eq:referee_bl}
|\heat_t^\star(\mu_0-\mu_1)|(X)
\le
\tcc(t)
\,
\bl^\star(\mu_0,\mu_1)
\quad
\text{for all}\ t>0.
\end{equation} 
Moreover, if $\mu\in\meas_+(X)$ and $t>0$, then $\heat_t^\star\mu\in\meas^{\rm ac}_+(X)$ with $\heat_t^\star\mu(X)\le \mu(X)$ and 
\begin{equation}
\label{eq:dual_heat_bl}
\int_Xf\di\heat_t^\star\mu
=
\int_X\heat_tf\di\mu
\quad
\text{for all}\ f\in\con_b(X).
\end{equation}
\end{theorem}

\begin{proof}
Given $f_0,f_1\in\leb^1(X)$ with $f_0,f_1\ge0$, by~\eqref{eq:heat_in_L_infty} and~\eqref{eq:ltbl} we can estimate 
\begin{equation*}
\int_Xg\,\heat_t(f_0-f_1)\di\m
=
\int_X\heat_tg\,(f_0-f_1)\di\m
\le 
\|\heat_tg\|_{\bl}
\,\bl^\star(f_0\m,f_1\m)
\le 
\tcc(t)
\,\bl^\star(f_0\m,f_1\m)
\end{equation*}
for all $g\in\leb^\infty(X)$ such that $\|g\|_{\leb^\infty}\le1$ and $t>0$, yielding~\eqref{eq:referee_bl_ac}.
Therefore, for each $t>0$, the map $\heat_t^\star$ is $\tcc(t)$-Lipschitz from  $\meas^{\rm ac}_+(X)$ endowed with the $\bl^\star$ distance to $\meas_+(X)$ endowed with the total-variation distance.
Since the latter metric space is complete, and owing to \cref{res:density_ac}, $\heat_t^\star$ can be extended to a $\tcc(t)$-Lipschitz map (for which we keep the same noation) from  $\meas_+(X)$ endowed with the $\bl^\star$ distance to $\meas_+(X)$ endowed with the total-variation distance, proving~\eqref{eq:referee_bl}.
To conclude, it is enough to observe that, given $\mu\in\meas_+(X)$ and $t>0$, the measure $\heat_t^\star\mu\in\meas_+(X)$ uniquely defined as above must coincide with the unique element of $\mathrm{rba}(X)$ given by \cref{res:dual_heat}.
\end{proof}

By \cref{res:referee_bl}, for each $x\in X$ there exists a non-negative $\mathsf{h}_t[x]\in\leb^1(X)$ such that
\begin{equation}
\label{eq:density_h}
\heat_t^\star\delta_x=\mathsf h_t[x]\,\m,
\quad
\text{for all}\
t>0.
\end{equation}
Therefore, according to~\eqref{eq:dual_heat_bl}, if $f\in\con_b(X)$, then
\begin{equation}
\label{eq:new_heat_cb}
\heat_tf(x)
=
\int_X f\,\mathsf h_t[x]\di\m
\quad
\text{for all}\
t>0.
\end{equation}

The following result collects the basic properties of the density $\mathsf h_t[\,\cdot\,]$.
Its proof is very similar to that of~\cite{G22}*{Lem.~3.24} and is thus omitted.

\begin{corollary}
\label{res:density_heat}
Let $t>0$.
The following hold:
\begin{enumerate}[label=(\roman*),itemsep=.5ex,topsep=.5ex,leftmargin=4ex]

\item
\label{item:density} 
$\heat_s(\mathsf h_t[x])=\mathsf h_{s+t}[x]$ $\m$-a.e.\ in $X$, for each $x\in X$ and $s\ge0$;

\item
$\mathsf h_t[x](y)=\mathsf h_t[y](x)$ for $\m$-a.e.\ $x,y\in X$.

\end{enumerate}
\end{corollary}

As a simple application of \cref{res:density_heat}, we obtain the following result.

\begin{corollary}
\label{res:new_heat_bdd}
Formula~\eqref{eq:new_heat_cb} holds for every $f\in\leb^\infty(X)$.
\end{corollary}

\begin{proof}
Given $\eps>0$ and $f\in\leb^\infty(X)$, by~\eqref{eq:c-Feller} we get that $\heat_\eps f\in\lip_b(X)$.
We also note that, given $t>0$, again by~\eqref{eq:c-Feller}, the maps $(\heat_{t+\eps}f)_{\eps>0}$ are equi-bounded and equi-Lipschitz in~$X$, and thus $\heat_{t+\eps}f(x)\to\heat_tf(x)$ for every $x\in X$ as $\eps\to0^+$.
Therefore, since $\heat_\eps f\to f$ weakly$^\star$ in $\leb^\infty(X)$ as $\eps\to0^+$, from \cref{res:density_heat}\ref{item:density} and~\eqref{eq:new_heat_cb} we get that 
\begin{equation*}
\int_X f\,\mathsf h_t[x]\di\m
=
\lim_{\eps\to0^+}
\int_X\heat_\eps f\,\mathsf h_t[x]\di\m
=
\lim_{\eps\to0^+}
\heat_t(\heat_\eps f)(x)
=
\lim_{\eps\to0^+}
\heat_{t+\eps} f(x)
=
\heat_t f(x)
\end{equation*}
for every $x\in X$, concluding the proof.
\end{proof}

\begin{remark}
\cref{res:dual_heat,res:density_heat,res:referee_bl,res:new_heat_bdd} have been obtained under the $\lip$-to-$\lip$ contractivity property~\eqref{eq:c-lip-to-lip}, see~\cite{AGS15}*{Prop.~3.2} and~\cite{G22}*{Prop.~3.13}.
We recall that, by \cref{res:diam_ltl}, property~\eqref{eq:c-lip-to-lip} is weaker than~\eqref{eq:c-Feller} in bounded metric spaces.
\end{remark}

\subsection{\texorpdfstring{$\wa_1$-$\leb^1$}{W1-L1} regularization}

In the following result, we provide a comparison between $\leb^1$ and $\wa_1$ distances of non-negative functions.

\begin{theorem}
\label{th:W1-He1}
If $f_0,f_1\in \leb^1(X)$ with $f_0,f_1\ge0$, then
\begin{equation}
\label{eq:W1-He1}
\|\heat_t(f_0-f_1)\|_{\leb^1}
\le 
\cc(t)
\,
\wa_1(f_0\m,f_1\m)
\quad
\text{for all}\ t>0.
\end{equation}
In addition, if $\mu_0,\mu_1\in\mathscr{P}_1(X)$, then
\begin{equation}
\label{eq:W1-totvar}
|\heat^\star_t(\mu_0-\mu_1)|(X)
\le
\cc(t) \,
\wa_1(\mu_0,\mu_1)
\quad
\text{for all}\ t>0,
\end{equation}
and so, as a consequence, 
\begin{equation}
\label{eq:W1-density_h}
\|\mathsf h_t[x]-\mathsf h_t[y]\|_{\leb^1}
\le 
\cc(t)\,\d(x,y)
\quad
\text{for all}\ x,y\in X,\ t>0.
\end{equation}
\end{theorem}

\begin{proof}
Let $t>0$ be fixed.
Given $g\in \leb^\infty(X)$, by~\eqref{eq:heat_in_L_infty}, \eqref{eq:dual_W1} and~\eqref{eq:c-Feller}, 
we can estimate 
\begin{equation*}
\begin{split}
\int_X g\,\heat_t(f_0-f_1)\di\m
&=
\int_X \heat_tg\di(f_0\m-f_1\m)
\le 
\lip(\heat_tg)\,\wa_1(f_0\m,f_1\m)
\\
&\le 
\cc(t)\,\|g\|_{\leb^\infty}
\,\wa_1(f_0\m,f_1\m),
\end{split}
\end{equation*}
readily yielding~\eqref{eq:W1-He1}.
To prove~\eqref{eq:W1-totvar}, we argue as follows.
By~\cite{V09}*{Th.~6.18}, we can find $\mu_0^k=f_0^k\m$ and $\mu_1^k=f_1^k\m$ in $\mathscr P_1(X)$, with $k\in\N$, such that
\begin{equation}
\label{eq:conv_wa1}
\lim_{k\to\infty}
\wa_1(\mu_0^k,\mu_0)
=
\lim_{k\to\infty}
\wa_1(\mu_1^k,\mu_1)
=
0.
\end{equation}
By~\eqref{eq:W1-He1}, we know that
\begin{equation*}
\|\heat_t(f_0^k-f_1^k)\|_{\leb^1}
\le 
\cc(t)
\,
\wa_1(\mu_0^k,\mu_1^k)
\quad
\text{for all}\ 
k\in\N.
\end{equation*}
Given $g\in\con_b(X)$ with $\|g\|_{\leb^\infty}\le1$, by~\eqref{eq:heat_in_L_infty} and~\eqref{eq:dual_heat} we can estimate
\begin{equation*}
\|\heat_t(f_0^k-f_1^k)\|_{\leb^1}
\ge 
\int_X  g\,\heat_t(f_0^k-f_1^k)\di\m
=
\int_X \heat_t g\di(\mu_0^k-\mu_1^k),
\end{equation*}
so that 
\begin{equation*}
\int_X \heat_t g\di(\mu_0^k-\mu_1^k)
\le 
\cc(t)
\,
\wa_1(\mu_0^k,\mu_1^k)
\quad
\text{for all}\ 
k\in\N
\end{equation*}
whenever $g\in\con_b(X)$ with $\|g\|_{\leb^\infty}\le1$.
Thanks to~\eqref{eq:c-Feller}, $\heat_t g\in\con_b(X)$.
Thus, recalling~\cite{V09}*{Th.~6.9},  we can exploit~\eqref{eq:conv_wa1} to pass to the limit as $k\to\infty$ and get that
\begin{equation*}
\int_X \heat_t g\di(\mu_0-\mu_1)
\le 
\cc(t)
\,
\wa_1(\mu_0,\mu_1)
\end{equation*}
whenever $g\in\con_b(X)$ with $\|g\|_{\leb^\infty}\le1$.
Recalling the definition in~\eqref{eq:dual_heat}, we get that
\begin{equation*}
\int_X  g\di\heat_t^\star(\mu_0-\mu_1)
\le 
\cc(t)
\,
\wa_1(\mu_0,\mu_1)   
\end{equation*}
whenever $g\in\con_b(X)$ with $\|g\|_{\leb^\infty}\le1$, readily yielding~\eqref{eq:W1-totvar}.
The validity of~\eqref{eq:W1-density_h} hence easily follows by recalling the definition of $\mathsf h_t[\,\cdot\,]$ in~\eqref{eq:density_h} and applying~\eqref{eq:W1-totvar} to $\mu_0=\delta_x$ and $\mu_1=\delta_y$, $x,y\in X$, completing the proof.
\end{proof}

\begin{remark}
In $\RCD(K,\infty)$ spaces, \cref{th:W1-He1} has been proved in~\cite{AGS14}*{Cor.~6.6}.
It is worth observing that~\eqref{eq:W1-totvar} can be equivalently rephrased in terms of the \emph{$1$-Matusita--Hellinger distance}.
We refer to~\cite{LuSa21}*{Th.~5.2.} and to~\cite{DF22} for similar inequalities.
\end{remark}

\subsection{Quantitative \texorpdfstring{$\leb^2$}{L2} contraction estimate}

From now on, in all the subsequent results, we additionally assume that 
\begin{equation}
\label{eq:cc_int_loc}
\cc\in \leb^1_{\rm loc}([0,+\infty))
\end{equation}
and we define $\CC\colon[0,\infty)\to [0,\infty)$ by letting 
\begin{equation}\label{eq:C(t)}
\CC(t)=\int_0^t \cc(s)\di s
\quad
\text{for all}\
t\ge0.
\end{equation}
We warn the reader that~\eqref{eq:cc_int_loc} holds in the settings considered in \cref{sec:explicit bound,sec:examples}.

The following result, generalizing~\cite{BaLe96}*{Th.~4.1}, provides a quantification of the $\leb^2$ contraction property~\eqref{eq:contraction} of the heat semigroup on sufficiently smooth functions.

\begin{theorem}
\label{res:quant_contraction}
If $f\in\sob^{1,2}(X)\cap \leb^\infty(X)$ is such that $|\slope f|_w\in\leb^1(X)$, then 
\begin{equation*}
\|f\|_{\leb^2}^2
-
\|\heat_{t/2}f\|_{\leb^2}^2
\le 
\CC(t)\,\|f\|_{\leb^\infty}\||\slope f|_w\|_{\leb^1}
\quad  
\text{for all}\ t\ge0.
\end{equation*}
\end{theorem}

\begin{proof}
Taking $g=f$ in \cref{res:heat_ibp} and using~\eqref{eq:c-Feller}, we can estimate
\begin{equation*}
\begin{split}
\int_X f(f-\heat_tf)\di\m
&= 
\int_0^t\int_X\slope f\cdot\slope\heat_sf\di\m\di s
\le 
\int_0^t|\slope f|_w\,\lip(\heat_s f)\di\m\di s
\\
&\le 
\int_0^t|\slope f|_w\,\cc(s)\,\|f\|_{\leb^\infty}\di\m\di s
=
\CC(t)\,\|f\|_{\leb^\infty}\,\int_X|\slope f|_w\di\m
\end{split}
\end{equation*}
and the conclusion follows by the symmetry and the semigroup property of $(\heat_t)_{t\ge0}$.
\end{proof}

As~\cite{BaLe96}*{Th.~4.1}, \cref{res:quant_contraction} can be refined provided that the heat semigroup $(\heat_t)_{t\ge0}$ is \emph{$\theta$-ultracontractive}, i.e., for some Borel function $\theta\colon(0,\infty)\to(0,\infty)$, it holds that
\begin{equation}
\label{eq:ultrac}
\|\heat_t f\|_{\leb^\infty}
\le
\theta(t)\,
\|f\|_{\leb^1}
\quad
\text{for all}\ t>0.
\end{equation}
Precisely, we get the following interpolation inequality for bounded $\bv$ functions.

\begin{corollary}
Under~\eqref{eq:ultrac}, if $f\in\bv(X)\cap\leb^\infty(X)$, then
\begin{equation}
\label{eq:isostrana}
\|f\|_{\leb^2}^2
\le 
\inf_{t>0}
\left(
\theta(\tfrac t2)\,\|f\|_{\leb^1}^2
+
\CC(t)\,\|f\|_{\leb^\infty}\var(f)
\right).
\end{equation}
\end{corollary}

\begin{proof}
We begin by observing that, by~\eqref{eq:contraction} and~\eqref{eq:ultrac}, 
\begin{equation*}
\|\heat_{t/2}
f\|_{\leb^2}^2
\le 
\|\heat_{t/2}
f\|_{\leb^1}
\,
\|\heat_{t/2}
f\|_{\leb^\infty}
\le 
\theta(\tfrac t2)\,\|f\|_{\leb^1}^2.
\end{equation*}
Owing to \cref{res:quant_contraction}, we hence plainly get that 
\begin{equation}
\label{eq:volpe}
\|f\|_{\leb^2}^2
\le 
\theta(\tfrac t2)\,\|f\|_{\leb^1}^2
+
\CC(t)\,\|f\|_{\leb^\infty}\||\slope f|\|_{\leb^1}
\quad
\text{for all}\ t>0,
\end{equation}
whenever $f\in\lip_{bs}(X)$.
In view of~\eqref{eq:var}, we can find $(f_k)_{k\in\N}\subset\lip_{bs}(X)$ such that $f_k\to f$ in $\leb^1(X)$ as $k\to\infty$ and 
\begin{equation*}
\var(f)
=
\lim_{k\to\infty}
\int_X|\slope f_k|\di\m.
\end{equation*}
Up to a truncation, we can also assume that $\|f_k\|_{\leb^\infty}\le\|f\|_{\leb^\infty}$ for $k\in\N$.
The conclusion hence follows by applying~\eqref{eq:volpe} to each $f_k$ and then passing to the limit as $k\to\infty$. 
\end{proof}

\begin{remark}
In a $\RCD(0,N)$ space, it holds that $\theta(t)\le C_Nt^{-\frac N2}$ and $\CC(t)\le C_N\sqrt{t}$ for all $t>0$ for some $C_N>0$ depending on $N$ only (that may vary from line to line in what follows), see \cref{rem:ultracontractivity} below. 
In this setting, inequality~\eqref{eq:isostrana} hence implies that
\begin{equation}
\label{eq:isso}
\|f\|_{\leb^2}^2
\le 
C_N\,\|f\|_{\leb^1}^{\frac2{N+1}}\,\|f\|_{\leb^\infty}^{\frac N{N+1}}\,\var(f)^{\frac N{N+1}}
\end{equation}
whenever $f\in\bv(X)\cap\leb^\infty(X)$. 
If $f=\chi_E$ for some $E\subset X$ such that $\chi_E\in\bv(X)$, then~\eqref{eq:isso} becomes 
$\m(E)^{\frac {N-1}N} 
\le C_N\,\per(E)$.
For this reason, the interpolation inequality~\eqref{eq:isostrana} is a kind of implicit isoperimetric-type inequality.
\end{remark}

\begin{remark}
\label{rem:ultracontractivity}
The ultracontractivity property~\eqref{eq:ultrac} is available in a wide range of settings, such as Markov spaces supporting a Sobolev inequality~\cite{BGL14}*{Sect.~6.3}, hence $\RCD(K,N)$ spaces with $N<\infty$~\cite{GM14}*{Rem.~5.17} and sub-Riemannian manifolds~\cite{GN96}.
For $\RCD(K,\infty)$ spaces with a uniform lower bound on the  measure of balls, see~\cite{DeMoSe}*{Prop.~2.4}.
\end{remark}

\subsection{Caloric-type Poincaré inequality and compactness}

The following result gives a \emph{caloric-type Poincaré inequality} for $\bv$ functions.

\begin{theorem}
\label{res:L1-grad}
If $f\in\bv(X)$, then
\begin{equation}
\label{eq: L1-grad}
\|f-\heat_tf\|_{\leb^1}
\le 
\CC(t)
\,\var(f)
\quad 
\text{for all}\ t\ge0.
\end{equation}
\end{theorem}

\begin{proof}
We can find $f_k\in\lip_{bs}(X)$ such that $f_k\to f$ in $\leb^1(X)$ as $k\to\infty$ and
\begin{equation}
\label{eq:approx_BV}
\var(f)(X)
=
\lim_{k\to+\infty}\int_X|\slope f_k|\di\m.
\end{equation}
In particular, $f_k\in \sob^{1,2}(X)$ with $|\slope f_k|_w\le|\slope f_k|$ $\m$-a.e.\ in~$X$.
Now, given $g\in \sob^{1,2}(X)\cap \leb^\infty(X)$, by \cref{res:heat_ibp} we can estimate 
\begin{equation*}
\int_X g\,(f_k-\heat_tf_k)\di\m
=
\int_0^t\int_X
\slope f_k\cdot \slope\heat_sg\di\m\di s
\le 
\||\slope f_k|\|_{\leb^1}
\int_0^t\,\||\slope\heat_s g|_w\|_{\leb^\infty}\di s.
\end{equation*}
Since $\heat_sg\in\lip_b(X)$ with $|\slope\heat_sg|_w\le|\slope\heat_sg|\le \cc(s)\,\|g\|_{\leb^\infty}$ for all $s\in(0,t)$ thanks to~\eqref{eq:c-Feller}, we can write
\begin{equation*}
\int_X g\,(f_k-\heat_tf_k)\di\m
\le 
\||\slope f_k|\|_{\leb^1}
\,
\|g\|_{\leb^\infty}
\int_0^t\,\cc(s)\di s
=
\CC(t)\,\|g\|_{\leb^\infty}\,\||\slope f_k|\|_{\leb^1}
\end{equation*}
whenever $g\in \sob^{1,2}(X)\cap \leb^\infty(X)$.
Now, given $g\in \leb^\infty(X)$, by a plain approximation argument exploiting~\cite{G22}*{Lem.~3.2}, we can find $g_j\in \sob^{1,2}(X)\cap \leb^\infty(X)$ such that $g_n\overset{\ast}{\rightharpoonup} g$ in $\leb^\infty(X)$.
Consequently, we get that 
\begin{equation*}
\int_X g\,(f_k-\heat_tf_k)\di\m
\le 
\CC(t)\,\|g\|_{\leb^\infty}\,\||\slope f_k|\|_{\leb^1}
\end{equation*}  
whenever $g\in \leb^\infty(X)$. 
The conclusion hence readily follows by~\eqref{eq:approx_BV}.
\end{proof}

As a consequence of \cref{res:L1-grad}, we can prove the following compactness result for uniformly bounded $\bv$ functions.

\begin{corollary}[Compactness]
Let $(X,\d)$ be a proper metric space.
If $(f_k)_{k\in\N}\subset\bv(X)$ is such that
\begin{equation*}
\sup_{k\in\N}\|f_k\|_{\leb^\infty}+\var(f_k)<\infty
\end{equation*} 
then there exists a subsequence $(f_{k_j})_{j\in\N}$ and $f\in \leb^1_{\rm loc}(X)$ such that $f_{k_j}\to f$ in $\leb^1_{\rm loc}(X)$.
\end{corollary}

\begin{proof}
Define $f_{k,n}=\heat_{\frac1n}f_k$ for $k,n\in\N$ and note that, in virtue of~\eqref{eq:c-Feller}, $f_{k,n}\in\lip_b(X)$ with $\|f_{k,n}\|_{\leb^\infty}\le M$ and $\lip(f_{k,n})\le\cc\left(\frac1n\right)M$, where $M=\sup_{k\in\N}\|f_k\|_{\leb^\infty}<\infty$.
In particular, for each $n\in\N$ fixed, the sequence $(f_{k,n})_{k\in\N}\subset\lip_b(X)$ is equi-bounded and equi-Lipschitz.
By Arzelà--Ascoli's Theorem, we can thus find a sequence $(k_j)_{j\in\N}$ such that $(f_{k_j,n})_{j\in\N}$ is uniformly convergent on any bounded $U\subset X$.
Consequently, we can exploit \cref{res:L1-grad} to estimate
\begin{equation*}
\begin{split}
\limsup_{i,j\to\infty}
\int_U|f_{k_i}-f_{k_j}|\di\m 
&\le 
\limsup_{i,j\to\infty}
\int_U|f_{k_i,n}-f_{k_j,n}|\di\m 
\\
&\quad+
\limsup_{i,j\to\infty}
\int_U|f_{k_i}-f_{k_i,n}|+|f_{k_j}-f_{k_j,n}|\di\m
\\
&\le 
2\,\CC\left(\tfrac1n\right)\,\sup_{k\in\N}\var(f_k) 
\end{split}
\end{equation*} 
for any bounded $U\subset X$.
Since $n\in\N$ is arbitrary and $L^1(U)$ is a Banach space, this proves that $(f_{k_j})_{j\in\N}$ converges in $L^1(U)$ for any bounded $U\subset X$.
Up to extracting a further subsequence (which we do not relabel for simplicity), we can find $f\in\leb^1_{\rm loc}(X)$ such that $f_{k_j}\to f$ in $\leb^1_{\rm loc}(X)$, yielding the conclusion.
\end{proof}

By combining \cref{th:W1-He1,res:L1-grad}, we get the following interpolation estimate for the $L^1$ norm of a $\bv$ function.

\begin{corollary}
\label{th:W1-grad-L1}
If $f\in \mathrm{BV}(X)$, then
\begin{equation}
\label{eq:W1-grad-L1}
\|f\|_{\leb^1}
\le 
\cc(t)\,\wa_1(f^{+}\m,f^{-}\m)+\CC(t)\,\var(f)
\quad  
\text{for all}\ t>0.
\end{equation}
\end{corollary}

\begin{proof}
By \cref{th:W1-He1,res:L1-grad}, we can estimate
\begin{equation*}
\begin{split}
\cc(t)\,\wa_1(f^+\m,f^-\m)
&\ge 
\|\heat_t(f^+-f^-)\|_{\leb^1}
\ge
\|f\|_{\leb^1}
-
\|f-\heat_tf\|_{\leb^1}
\\
&\ge 
\|f\|_{\leb^1}
-
\CC(t)\,\var(f)
\end{split}
\end{equation*}
readily yielding the conclusion.
\end{proof}

\subsection{Implicit indeterminacy estimate}

The next result provides an implicit indeterminacy estimate, which, in few words, quantifies the relation between the Wasserstein distance of positive and negative parts of an $\leb^1\cap\leb^\infty$ function and the size of its zero set.

\begin{theorem}
\label{th:impl_indet}
If $\m(X)<\infty$ and $f\in \leb^{\infty}(X,\m)$, then
\begin{equation}
\label{eq:wa_1-t_ind}
\|f\|_{\leb^1}
\le 
\cc(t)\,\wa_1(f^+\m,f^-\m)
+
2\sqrt{\CC(t)\,\|f\|_{\leb^{\infty}}\,\|f\|_{\leb^{1}}\,\per(\{f>0\})}
\quad
\text{for all}\ t\ge0.
\end{equation}
\end{theorem}

To prove \cref{th:impl_indet}, we need the following preliminary result.

\begin{lemma}\label{th:per}
If $A\subset X$ is an $\m$-measurable set with $\m(A)<\infty$, then
\begin{equation}
\label{eq:per_set}
\int_{A^c}{\heat_t(\chi_{A})}\di\m
\le
\frac{1}{2}\,\CC(t)\,\per(A)
\quad
\text{for all}\
t\ge0.
\end{equation}
Moreover, if $\m(X)<\infty$ and $f\in \leb^{\infty}(X)$, then
\begin{equation}
\label{eq:per_Linfty_funct}
\int_{X}\sqrt{\heat_t(f^+)\,\heat_t(f^{-})}\di\m
\le
\sqrt{
\CC(t)
\,
\|f\|_{\leb^{\infty}}
\,
\|f\|_{\leb^{1}}
\,
\per(\lbrace f>0\rbrace)}
\quad
\text{for all}\ t\ge0.
\end{equation}
\end{lemma}

\begin{proof}
Since it is enough to discuss the case $\chi_A\in\bv(X)$, from \cref{res:L1-grad} we immediately get
\begin{align*}
C(t)\,\per(A)
&\ge 
\|\chi_A-\heat_t(\chi_A)\|_{\leb^1}=
\int_A (1-\heat_t(\chi_A))\di\m 
+ 
\int_{A^c}\heat_t(\chi_A)\di\m \\
&= 
\int_X (1-\heat_t(\chi_A))\di\m
-
\int_{A^c}1\di\m
+
2\int_{A^c}\heat_t(\chi_A)\di\m
=
2\int_{A^c}\heat_t(\chi_A)\di\m,
\end{align*}
yielding~\eqref{eq:per_set}.
Concerning~\eqref{eq:per_Linfty_funct}, since $f^-\le \|f^-\|_{\leb^\infty}\chi_{\{f\le 0\}}$, by~\eqref{eq:maxprinciple}, the Cauchy--Schwarz inequality, the mass-preservation property~\eqref{eq:mass_preserving} and the previous~\eqref{eq:per_set}, we get
\begin{align*}
\left(
\int_{\lbrace f>0\rbrace}\sqrt{\heat_t(f^+)\,\heat_t(f^{-})}\di\m
\right)^2
&\leq 
\|f^-\|_{\leb^{\infty}}
\left(
\int_{\lbrace f>0\rbrace}\sqrt{\heat_t(f^+)\,\heat_t(\chi_{\lbrace f\leq 0\rbrace})}\di\m 
\right)^2
\\
&\leq 
\|f^-\|_{\leb^{\infty}}
\,
\|H_{t}(f^+)\|_{\leb^1}
\int_{\lbrace f>0\rbrace}{\heat_t(\chi_{\lbrace f\leq 0\rbrace})}\di\m
\\
&\leq
\frac12
\,
\CC(t)
\,
\|f^-\|_{\leb^{\infty}}
\,
\|f^+\|_{\leb^1}
\,
\per(\lbrace f>0\rbrace).
\end{align*}
Similarly, we can also estimate
\begin{equation*}
\left(
\int_{\lbrace f\le0\rbrace}\sqrt{\heat_t(f^+)\,\heat_t(f^{-})}\di\m
\right)^2
\leq 
\frac12
\,
\CC(t)
\,
\|f^+\|_{\leb^{\infty}}
\,
\|f^-\|_{\leb^1}
\,
\per(\lbrace f\le0\rbrace),
\end{equation*}
and the conclusion readily follows by observing that $\per(\lbrace f>0\rbrace)=\per(\lbrace f\leq 0\rbrace)$, $\|f^\pm\|_{\leb^{\infty}}\leq \|f\|_{\leb^{\infty}}$ and $\|f^+\|_{\leb^1}+\|f^-\|_{\leb^1}=\|f\|_{\leb^1}$.
\end{proof}

We can now give the proof of \cref{th:impl_indet}.

\begin{proof}[Proof of \cref{th:impl_indet}]
By~\eqref{eq:W1-He1}, we have
\begin{equation}
\label{eq:W1-he2_ind}
\|\heat_t(f^+-f^-)\|_{\leb^1}
\le 
\cc(t)
\,
\wa_1(f^+\m,f^-\m).
\end{equation}
Since $|a-b| \ge a+b-2\sqrt{ab}$ whenever $a,b\ge0$, from~\eqref{eq:per_Linfty_funct} we get
\begin{equation}
\label{eq:he2-per_ind}
\begin{split}
\|\heat_t(f^+-f^-)\|_{\leb^1}
&\ge 
\int_X \heat_t(f^+)+\heat_t(f^-)-2\sqrt{\heat_t(f^+)\heat_t(f^-)}\di\m
\\
&\ge 
\|f\|_{\leb^1}
-
2\sqrt{\CC(t)\,\|f\|_{\leb^{\infty}}\,\|f\|_{\leb^{1}}\,\per(\{f>0\})}
\end{split}
\end{equation}
and the conclusion follows by combining~\eqref{eq:W1-he2_ind} and \eqref{eq:he2-per_ind}.
\end{proof}

\subsection{Implicit estimates for eigenfunctions}

\cref{th:W1-He1,th:impl_indet} can be exploited to obtain implicit lower bounds on the (perimeter of the) nodal set and an indeterminacy-type inequality for eigenfunctions.

\begin{theorem}
\label{th:per nodal set}
If $f_{\lambda}$ is a $\lambda$-eigenfunction, then 
\begin{equation}\label{eq:per_nodal_set}
\per(\{f_{\lambda}>0\})
\,
\|f_{\lambda}\|_{\leb^{\infty}}
\ge 
\frac{(1-e^{-\lambda t})^2}{4\,\CC(t)}
\,
\|f_{\lambda}\|_{\leb^1}
\quad 
\text{for all}\ t>0
\end{equation}
and 
\begin{equation}
\label{eq:W1_eigen_impl}
\wa_{1}(f_{\lambda}^{+}\m,f_{\lambda}^{-}\m)
\ge 
\frac{e^{-\lambda t}}{\cc(t)}\,\|f_{\lambda}\|_{\leb^{1}}
\quad 
\text{for all}\ t>0.
\end{equation}
\end{theorem}

\begin{proof}
The proof of~\eqref{eq:per_nodal_set} is the same of~\eqref{eq:wa_1-t_ind}, since  one just need to replace~\eqref{eq:W1-he2_ind} with
\begin{equation}
\label{eq:norm1heat}
\|\heat_t(f^+_\lambda-f^-_\lambda)\|_{\leb^1}
=
\|\heat_tf_\lambda\|_{\leb^1}
=
e^{-\lambda t}\,\|f_{\lambda}\|_{\leb^{1}} 
\end{equation}
by \cref{res:heated_eigenf}.
Inequality~\eqref{eq:W1_eigen_impl} is again a consequence of  \cref{res:heated_eigenf}, together with \cref{th:W1-He1}.
\end{proof}

One can get rid of the $\leb^\infty$ norm in the lower bound~\eqref{eq:per_nodal_set} as soon as the heat semigroup $(\heat_t)_{t\ge0}$ is $\theta$-ultracontractive
as in~\eqref{eq:ultrac}.
Precisely, we have the following result. 

\begin{corollary}
\label{res:ultrac_nodal}
Under~\eqref{eq:ultrac}, if $f_\lambda$ is a $\lambda$-eigenfunction, then
\begin{equation}
\label{eq:ultrac_nodal}
\per(\left\{f_\lambda>0\right\})
\ge
\sup_{t>0}
\frac{e^{-\lambda t}(1-e^{-\lambda t})^2}{4\,\theta(t)\,\CC(t)}.
\end{equation}
\end{corollary}

\begin{proof}
Thanks to \cref{res:heated_eigenf} and~\eqref{eq:ultrac}, we can estimate
\begin{equation*}
\|f_\lambda\|_{\leb^\infty}
=
e^{\lambda t}\|\heat_t f_\lambda\|_{\leb^\infty}
\le 
e^{\lambda t}\,\theta(t)\,\|f_\lambda\|_{\leb^1}
\quad
\text{for all}\ t>0,
\end{equation*}
which, combined with~\eqref{eq:per_nodal_set}, easily yields~\eqref{eq:ultrac_nodal}.
\end{proof}

\subsection{Implicit Buser inequality}

We conclude this section with the following result, yielding an implicit Buser inequality for the Cheeger constants $h_0(X)$ and $h_1(X)$.

\begin{theorem}
\label{th: impl Buser}
The following hold:
\begin{enumerate}[label=(\roman*),itemsep=1ex]

\item 
\label{item:impl_buser_finite}
if $\m(X)<\infty$, then
$h_1(X)\ge\sup\limits_{t>0}\left\{\dfrac{1-e^{-\lambda_1t}}{\CC(t)}\right\}$;

\item
\label{item:impl_buser_infinite}

if 
$\m(X)=\infty$, then
$h_0(X)\ge 2\,\sup\limits_{t>0}\left\{\dfrac{1-e^{-\lambda_0t}}{\CC(t)}\right\}$.

\end{enumerate}
\end{theorem}

\begin{proof}
We start by observing that, by \cref{res:L1-grad}, we have 
\begin{equation}
\label{eq:per_meas-heat}
\begin{split}
\CC(t)\,\per(A)
&\geq 
\left\Vert \chi_A-\heat_t(\chi_A)\right\Vert_{\leb^1}
=
\int_A (1-\heat_t(\chi_A))\di\m+\int_{A^c}\heat_t(\chi_A)\di\m 
\\
&=
2\,
\m(A)-2\int_A \heat_t(\chi_A)\di\m 
= 
2\,
\m(A)-2\left\Vert \heat_{t/2}(\chi_A)\right\Vert^2_{\leb^2}
\end{split}
\end{equation}
for any $\m$-measurable set $A\subset X$, thanks to~\eqref{eq:maxprinciple}, \eqref{eq:mass_preserving},
\eqref{eq: heat self} and the semigroup property.
We prove the two statements separately.

\vspace{1ex}

\textit{Proof of~\ref{item:impl_buser_finite}}.
Assume $\m(X)=1$ without loss of generality. 
Since $\heat_t(1)=\m(X)=1$ because of~\eqref{eq:mass_preserving}, we immediately get that 
\begin{equation*}
\int_{X}  \heat_{t/2}(\chi_A-\m(A))\di\m=0.
\end{equation*}
We can hence apply \cref{res:heat_lambdas_10}\ref{item:heat_lambda_1} to get
\begin{equation}
\label{eq:chi_heat}
\left\Vert \heat_{t/2}(\chi_A)\right\Vert^2_{2}=\m(A)^2+\left\Vert \heat_{t/2}(\chi_A-\m(A))\right\Vert^2_{2}\leq \m(A)^2+e^{-\lambda_1 t}\left\Vert \chi_A-\m(A)\right\Vert^2_{2}.
\end{equation}
By direct computation, we can write 
\begin{equation*}
\left\Vert \chi_A-\m(A)\right\Vert^2_{2}=\m(A)\,(1-\m(A)),
\end{equation*}
so that, by combining~\eqref{eq:per_meas-heat} with~\eqref{eq:chi_heat}, we get that
\begin{equation}
\label{eq:per-meas}
\CC(t)\,\per(A)
\geq 
2\,\m(A)\,(1-\m(A))\,\big(1-e^{-\lambda_1t}\big) 
\quad 
\text{for every}\ t>0.
\end{equation}
The conclusion hence follows by recalling the definition in~\eqref{eq:Cheeger_1}.

\vspace{1ex}

\textit{Proof of~\ref{item:impl_buser_infinite}}.
We can bound the last term in the chain~\eqref{eq:per_meas-heat} using \cref{res:heat_lambdas_10}\ref{item:heat_lambda_0}.
The conclusion hence immediately follows by the definition in~\eqref{eq:Cheeger_0}.
\end{proof}

\begin{remark}
\label{rem:bl-replacing-w1}
By appealing to \eqref{eq:referee_bl} instead of \eqref{eq:W1-He1}, the inequalities in \cref{th:W1-grad-L1,th:impl_indet,th:per nodal set} can be alternatively formulated by involving the $\bl^\star$ distance instead of the $\wa_1$ distance (with the advantage that the former is not larger than the latter and always finite).
However, we have chosen to use the $\wa_1$ distance because the resulting inequalities can be directly compared with ones available in the literature.
\end{remark}

\section{Quantitative Lipschitz smoothing with controls and explicit bounds}
\label{sec:explicit bound}

\subsection{Quantitative Lipschitz smoothing with controls}

In the following, we give a power-type upper bound on the Lipschitz constant $\cc(t)$ in \cref{def:feller}.

We say that $\cc$ is \emph{controlled by the couple} $(M,b)\in [0,\infty)\times (0,1)$ if 
\begin{equation}\label{eq:controls}
\cc(t)\le \frac{M}{t^b} 
\quad 
\text{for all}\ t\in(0,1].
\end{equation}
Consequently, the primitive function $\CC(t)$ in~\eqref{eq:C(t)} is well defined for every $t\in[0,1]$ and, setting
$\widetilde{M}=\frac{M}{1-b}$, it satisfies  
\begin{equation}
\label{eq: bound C(t) very small times}                
\CC(t)
\le 
\widetilde{M}t^{1-b}
\quad 
\text{for all}\ t\in[0,1].
\end{equation}

The bound~\eqref{eq:controls} should be understood in an \emph{operative sense}, meaning that it allows us to obtain the inequalities in an explicit power-like form. 
In most of the cases, this analysis is enough, but in some specific situations---such as the Buser inequality in $\RCD(K,\infty)$ spaces with $K>0$~\cite{DeMo}---the \emph{exact} form of the function $\cc(t)$ allows to recover \emph{sharp} results (i.e., inequalities which are equalities in some non-trivial cases). 

In most of the relevant settings, the bound in~\eqref{eq:controls} holds with $b=\frac12$---see the discussion in \cref{sec:examples}, where a list of examples is presented.
This is, for instance, the case of $\RCD(K,\infty)$ spaces, in which the constant $M>0$ may depend on $K\in\R$.
Additionally, 
in \cref{sec:log-corr}, we discuss how our analysis can be also used to cover situations where $\cc$ is controlled by a more general \emph{po\-wer-lo\-ga\-rith\-mic} function (see \eqref{eq:log-controls} for the definition).

Last but not least, we remark that we do not explicitly compute the constants appearing in the results in order to keep the presentation short. 
However, the form of these constants can be easily obtained by following the computations suggested in the proofs.    

\begin{remark}\label{rem: t to tT}
We request the validity of the bound \eqref{eq:controls} up to time $t=1$ just to simplify the exposition of the proofs of \cref{res:expl_indet,res:expl_nodal,res:stein,th:expl Buser}.
However, the same statements still hold by requesting the validity of~\eqref{eq:controls} for all $t\in (0,T]$ for some $T>0$, with the constants appearing in the inequalities also depending on~$T$. 
To see this, it is sufficient to replace the choice of the optimal $t$ with $tT$ in all the proofs. 
\end{remark}

\subsection{Explicit indeterminacy estimate}

We begin with the following explicit version of the indeterminacy estimate in \cref{th:impl_indet}.

\begin{theorem}
\label{res:expl_indet}
If $\m(X)<\infty$ and $h_1(X)>0$, then, assuming \eqref{eq:controls} there exists a constant $C=C(M,b,h_1(X))>0$ such that 
\begin{equation}
\label{eq:indeterminacy}
\wa_1(f^+\m,f^-\m) \geq C \left(\frac{\|f\|_{\leb^{1}}}{\|f\|_{\leb^{\infty}}\per(\lbrace f > 0\rbrace)}\right)^{\frac{b}{1-b}}\|f\|_{\leb^{1}}
\end{equation}
for every $f\in \leb^{\infty}(X)$ satisfying $\displaystyle\int_X f\di\m=0$. 
\end{theorem}

\begin{proof}
We exploit~\eqref{eq:wa_1-t_ind} in combination with the choice 
\begin{equation}
\label{eq:scegli_t}
t=\theta \,\left(\frac{2\,\|f\|_{\leb^\infty}\per(\lbrace f> 0\rbrace)}{h_1(X)\,\|f\|_{\leb^1}}\right)^{\frac{1}{b-1}}
\end{equation}
where $\theta\in(0,1]$ has to be chosen later.
Note that $t\in(0,1]$ follows from~\eqref{eq:norm-cheeg}.
Hence
\begin{equation*}
\wa_1(f^+\m,f^-\m) 
\geq 
\frac{t^{b}\|f\|_{\leb^1}}{M}
\left(
1-
\theta^{\frac{1-b}{2}}
\sqrt{2\widetilde{M}h_1(X)}
\right)
\end{equation*}
and the conclusion follows from the definition in~\eqref{eq:scegli_t} by choosing $\theta$ sufficiently small.
\end{proof}

\begin{remark}
\label{rem:expl_indet}
Non-optimal indeterminacy estimates were considered in~\cites{Stein21,CaMaOr}.
In the class of closed Riemannian manifolds, the exponent $\frac{b}{1-b}$ in~\eqref{eq:indeterminacy} can be replaced by~$1$, and no smaller exponent is possible.
In this form, the inequality was proved in the more general setting of essentially non-branching $\mathsf{CD}(K,N)$ spaces with $N<\infty$ in~\cite{CaFa21} and in $\mathsf{RCD}(K,\infty)$ spaces in~\cite{DF22}.
Indeterminacy estimates with optimal exponent~$1$ and best possible multiplicative constant $C>0$ were recently achieved in~\cite{DuSa23} for spaces with simple geometry.
\end{remark}

\subsection{Explicit estimates for eigenfunctions}

We now provide an explicit version of the bounds given in \cref{th:per nodal set}.

We begin with the following explicit version of the first part of \cref{th:per nodal set}.

\begin{theorem}
\label{res:expl_nodal}
Assuming \eqref{eq:controls}, for every $\tilde{\lambda}>0$ there exists a constant $C=C(M,b,\tilde{\lambda})>0$ such that
\begin{equation}
\label{eq:per_eigen_explicit}
\per(\{f_{\lambda}>0\})
\ge 
C\lambda^{1-b}\frac{\|f_{\lambda}\|_{\leb^{1}(X)}}{\|f_{\lambda}\|_{\leb^{\infty}(X)}}
\end{equation}
for every $f_{\lambda}$ $\lambda$-eigenfunction with $\lambda\ge\tilde{\lambda}$.
\end{theorem}

\begin{proof}
It follows exploiting~\eqref{eq:per_nodal_set} for the admissible choice $t=\frac{\tilde{\lambda}}{\lambda}$. 
We omit the details.
\end{proof}

\begin{remark}
On an $N$-dimensional closed Riemannian manifold, inequality~\eqref{eq:per_eigen_explicit} can be coupled with the sharp bound $\|f_{\lambda}\|_{\leb^{\infty}}\le C\lambda^{\frac{N-1}{4}}\|f_{\lambda}\|_{\leb^1}$, see e.g.~\cite{SoZe11}, (here and below, $C>0$ is a constant independent of $\lambda$ which may vary from line to line) to recover the lower bound 
\begin{equation*}
\per(\{f_{\lambda}>0\})
\geq 
C\,\lambda^{\frac{3-N}{4}}
\end{equation*}
obtained in~\cites{CoMi10,SoZe11,Stein14}.
Noteworthy, our approach to establish the lower bound on the nodal set is different from the ones employed in~\cites{CoMi10,SoZe11,Stein14} and, up to our knowledge, is new.  
The sharp lower bound 
\begin{equation*}
\per(\{f_{\lambda}>0\})
\geq 
C\,\sqrt{\lambda}
\end{equation*}
conjectured by Yau has been proved in~\cite{Log18}. 
On compact $\mathsf{RCD}(K,N)$ spaces, one can instead exploit in~\eqref{eq:per_eigen_explicit} the bound $\|f_{\lambda}\|_{\leb^{\infty}}\le C\lambda^{\frac{N}{4}}\|f_{\lambda}\|_{\leb^1}$ obtained in~\cite{AHPT21}*{Prop.~7.1}, to achieve 
\begin{equation}\label{eq:nodal set expl RCD}
\per(\{f_{\lambda}>0\})
\geq
C\,\lambda^{\frac{1-N}{4}},
\end{equation}
which improves the previously best-known estimate given in~\cite{CaFa21}*{Th.~1.5}.
\cref{res:expl_nodal}, as well as \cref{res:ultrac_nodal}, provides a lower bound on the size of nodal sets in several sub-Riemannian structures, also see the recent work~\cite{EsLe23} for a related discussion.
\end{remark}

We can now move to the following explicit version of the second part of \cref{th:per nodal set}.

\begin{theorem}
\label{res:stein}
Assuming \eqref{eq:controls}, for every $\tilde{\lambda}>0$ there exists a constant $C=C(M,b,\tilde{\lambda})>0$ such that
\begin{equation}
\label{eq:W1_eigen_expl}
\wa_1(f^{+}_{\lambda}\m,f^{-}_{\lambda}\m)\ge \frac{C}{\lambda^{b}}\,\|f_{\lambda}\|_{\leb^{1}(X)}
\end{equation}
for every $f_{\lambda}$ $\lambda$-eigenfunction with $\lambda\ge\tilde{\lambda}$.
\end{theorem}

\begin{proof}
It follows exploiting~\eqref{eq:W1_eigen_impl} for the admissible choice $t=\frac{\tilde{\lambda}}{\lambda}$. 
We omit the details.
\end{proof}

\begin{remark}
\label{rem:stein}
In~\cite{Stein21}, it was conjectured that, on any closed Riemannian manifolds, there exist  some constants $C_2\ge C_1>0$ such that
\begin{equation}
\label{eq:W1_eigen_conj}
\frac{C_2}{\sqrt{\lambda}}
\,
\|f_{\lambda}\|_{\leb^{1}(X)} 
\ge
\wa_1(f^{+}_{\lambda}\m,f^{-}_{\lambda}\m)\ge 
\frac{C_1}{\sqrt{\lambda}}
\,
\|f_{\lambda}\|_{\leb^{1}(X)}.
\end{equation}
The left-hand side of~\eqref{eq:W1_eigen_conj} was confirmed in~\cite{CaMaOr}, while the right-hand side was established in~\cite{DF22} in the more general context of $\mathsf{RCD}(K,\infty)$ spaces (also see~\cite{Muk21} for an alternative proof of the right-hand side of~\eqref{eq:W1_eigen_conj} for closed Riemannian manifolds).
\cref{res:stein} yields the right-hand side of~\eqref{eq:W1_eigen_conj} any time~\eqref{eq:controls} holds $b=\frac12$. 
\end{remark}

\subsection{Explicit Buser inequality}

We now pass to the explicit version of the Buser inequalities provided in \cref{th: impl Buser}.

\begin{theorem}
\label{th:expl Buser}
Assuming \eqref{eq:controls}, there exist constants $C_{1,i}=C_{1,i}(M,b)>0$ and  $C_{2,i}=C_{2,i}(M,b)>0$, $i=0,1$, such that the following hold:
\begin{enumerate}[label=(\roman*),itemsep=1ex,topsep=1ex]

\item
\label{item:Buser_esplicit_1}
if $\m(X)<\infty$, then
$\lambda_1\le \max\left\{C_{1,1}h_1(X),C_{2,1}h_1(X)^{\frac{1}{1-b}}\right\}$;

\item
\label{item:Buser_esplicit_0}
if $\m(X)=\infty$, then
$\lambda_0\le \max\left\{C_{1,0}h_0(X),C_{2,0}h_0(X)^{\frac{1}{1-b}}\right\}$.
\end{enumerate}

\end{theorem}

\begin{proof}
We just prove~\ref{item:Buser_esplicit_1}, the other case~\ref{item:Buser_esplicit_0} being analogous.
Since $\m(X)<\infty$, we  apply \cref{th: impl Buser}\ref{item:impl_buser_finite}.
If $\lambda_1\ge 1$, then we choose $t=1/\lambda_1$ so that, recalling~\eqref{eq: bound C(t) very small times}, we find
\begin{equation}
\label{eq: Buser expl proof1}
h_1(X)\ge \frac{1-e^{-1}}{\widetilde{M}}\,\lambda_1^{1-b}.
\end{equation}
If $\lambda_1<1$ instead, then we simply choose $t=1$ and get
\begin{equation}\label{eq: Buser expl proof2}
h_1(X)
\ge 
\frac{1-e^{-\lambda_1}}{\widetilde{M}}
\,
=\frac{1-e^{-\lambda_1}}{\widetilde{M}\lambda_1}\lambda_1>\frac{(1-e^{-1})}{\tilde{M}}\lambda_1
\,
\end{equation}
since 
$r\mapsto \frac{1-e^{-r}}{r}$
is decreasing for $r\in(0,1]$. 
The conclusion thus follows by rearranging and combining~\eqref{eq: Buser expl proof1} with~\eqref{eq: Buser expl proof2}.
\end{proof}

\begin{remark}
Upper bounds on the first eigenvalue in terms of the Cheeger constant of the space were firstly proved in~\cite{Bus82} in the setting of closed Riemannian manifolds. 
An alternative proof based on heat semigroup techniques was given in~\cite{Led94}, and subsequently improved in~\cite{Led04} to a dimension-free estimate. 
The strategy of~\cites{Led94,Led04} was later refined in~\cite{DeMo}, yielding sharp estimates in $\mathsf{RCD}(K,\infty)$ spaces, with equality cases discussed in~\cite{DeMoSe}.
It is worth noticing that the lower bound $4\lambda_1\ge h^2_1(X)$ for $\m(X)<\infty$ (respectively, $4\lambda_0\ge h^2_0(X)$ for $\m(X)=\infty$) on the first eigenvalue in terms of the Cheeger constant was noticed independently by Maz'ja and Cheeger \cites{Maz62,C69}, and it is known to hold on any metric measure space~\cite{DeMo}*{Th.~4.2} and even in more general settings~\cite{FPSS24}*{Sect.~6.1}.
We also refer to~\cite{BK14}*{Sect.~3.4} for a lower bound in the sub-Riemannian context.
\end{remark}

\subsection{Explicit interpolation estimate}

We conclude the list of explicit results with the following explicit version of \cref{th:W1-grad-L1}. 
We need to reinforce \eqref{eq:controls} with a stronger control on $\cc$, that is, we require a power-type upper bound for \emph{all} times. 

\begin{theorem}
\label{th:W1-grad-L1 expl}
If $(\heat_t)_{t\ge 0}$ satisfies \eqref{eq:c-Feller} with $\cc$ such that
\begin{equation}
\label{eq:strong_control}
\cc(t)
\leq 
\frac{M}{t^b}
\quad
\text{for all}\ t>0
\end{equation} 
for some $(M,b)\in (0,\infty)\times (0,1)$,  then there exists $C=C(M,b)>0$ such that
\begin{equation}
\label{eq: W1-grad-L1 expl}
\|f\|_{\leb^1}
\le 
C\,\sob_1(f^{+}\m,f^{-}\m)^{1-b}\,\var(f)^b
\end{equation}
for every $f\in\bv(X)$. 
\end{theorem}

\begin{proof}
From~\eqref{eq:strong_control}, we immediately get that 
\begin{equation}
\label{eq:strong_control_2}
\CC(t)\leq\frac{\widetilde{M}}{t^{b-1}}
\quad
\text{for all}\ t>0,
\end{equation}
for some $\widetilde{M}=\widetilde{M}(M,b)>0$. 
Combining~\eqref{eq:strong_control} and~\eqref{eq:strong_control_2} with~\eqref{eq:W1-grad-L1}, we get that
\begin{equation*}
\|f\|_{\leb^1}
\le 
\frac{M}{t^{b}}\,\wa_1(f^{+}\m,f^{-}\m)+\frac{\widetilde{M}}{t^{b-1}}\,\var(f)
\quad 
\text{for all}\ t>0.
\end{equation*}
The inequality~\eqref{eq: W1-grad-L1 expl} hence follows by choosing $t=\frac{\wa_1(f^{+}\m,f^{-}\m)}{\var(f)}$.
\end{proof}

\begin{remark}
\label{rem:W1-grad-L1 expl}
In the context of smooth weighted Riemannian manifolds with non-negative weighted Ricci curvature, inequality~\eqref{eq: W1-grad-L1 expl} is essentially contained in~\cites{BoSh,BoShWa}. 
\end{remark}

\begin{remark}
Similarly as in \cref{rem:bl-replacing-w1}, by appealing to \eqref{eq:referee_bl} instead of \eqref{eq:W1-He1}, the inequalities in \cref{res:expl_indet,res:stein,th:W1-grad-L1 expl} can be alternatively formulated by involving the $\bl^\star$ distance instead of the $\wa_1$ distance.
However, we have chosen to use the $\wa_1$ distance because the resulting inequalities can be directly compared with ones available in the literature.
\end{remark}

\subsection{Logarithmic correction}\label{sec:log-corr}

We conclude this section by showing how the statements of \cref{res:expl_indet,res:expl_nodal,res:stein,th:expl Buser} can be modified if a po\-wer-lo\-ga\-rith\-mic upper bound on the Lipschitz constant $\cc(t)$ in \cref{def:feller} is at disposal.

We say that $\cc$ is \emph{controlled by the triplet} $(M,a,b)\in [0,\infty)^2\times (0,1)$ if 
\begin{equation}
\label{eq:log-controls}
\cc(t)\le M\,\frac{(1+|\log(t)|)^a}{t^b} 
\quad 
\text{for all}\ t\in(0,1].
\end{equation}

The following result collects some elementary estimates following from \eqref{eq:log-controls}.

\begin{lemma}
\label{res:controls_bounds}
Let $\cc$ be controlled by the triplet $(M,a,b)$.
For every $\eps>0$ there exists $T=T(\eps,a)\in(0,1)>0$, depending on $\eps$ and $a$ only, such that 
\begin{equation}\label{eq: log bound c(t) very small times} 
\cc(t)\le \frac{M}{t^{b+\eps}}
\quad 
\text{for all}\ t\in(0,T].
\end{equation}
Consequently, the primitive function $\CC$ is well defined and, setting
$\widetilde{M}=\frac{M}{1-b-\eps}$, it satisfies  
\begin{equation}
\label{eq: log bound C(t) very small times} 
\CC(t)
\le 
\widetilde{M}t^{1-b-\eps}
\quad 
\text{for all}\ t\in(0,T]\
\text{and}\ 
\eps<1-b.
\end{equation}
\end{lemma}

\begin{proof}
It follows by letting $T\in(0,1)$ be the smallest solution of
$(1+|\log(T)|)^a=T^{-\eps}$.
\end{proof} 

The bound~\eqref{eq: log bound c(t) very small times} allows us to get rid of the logarithmic term in~\eqref{eq:log-controls} and thus to directly apply the results of the previous section. Indeed, assuming~\eqref{eq:log-controls} instead of~\eqref{eq:controls}, it is immediate to realize that, for every $\eps>0$, the very same statements of \cref{res:expl_indet,res:expl_nodal,res:stein,th:expl Buser} hold with $b$ replaced by $b+\eps$, and with the constants depending also on~$T$ (cfr.~\cref{rem: t to tT}), and thus on $a$ and $\eps.$

Instead of appealing to \cref{res:controls_bounds}, one may directly work with the logarithmic term in~\eqref{eq:log-controls} along the computations.
However, this alternative route inevitably requires the use of special functions, which we have chosen to avoid for better clarity. 

Power-logarithmic bounds of the form~\eqref{eq:log-controls} with $a > 0$ are employed in some settings, such as  \emph{diamond fractals} (see~\cite{Alon21}*{Th.~5.6} and \cref{sec:other} below). 
However, to the top of our knowledge, it is not currently known whether these logarithmic bounds are sharp.

\section{Examples}

\label{sec:examples}

In this last section, we provide a brief overview of the settings where our results apply.

\subsection{Weak Bakry--\texorpdfstring{\'E}{É}mery condition}

\label{subsec:wbe}

Let $(X,\d,\m)$ be a metric-measure space satisfying the properties~\eqref{item:ass_d}, \eqref{item:ass_m} and~\eqref{item:ass_ch} listed in \cref{subsec:ass}.
Following~\cite{G22}*{Def.~3.4}, $(X,\d,\m)$ satisfies the \emph{weak Bakry--\'Emery condition} with respect to some Borel function $\kappa\colon[0,\infty)\to(0,\infty)$ such that $\kappa,\kappa^{-1}\in \leb^\infty_{\rm loc}([0,\infty))$, $\mathsf{BE}_w(\kappa,\infty)$ for short, if
\begin{equation}
\label{eq:wbe}
|\slope \heat_t f|_w^2
\le 
\kappa^2(t)\,
\heat_t(|\slope f|_w^2)
\quad
\text{$\m$-a.e.\ in}\ X
\end{equation} 
for all $f\in\sob^{1,2}(X)$ and $t>0$ (where $\kappa(0)=1$ for simplicity).
Adding the \emph{Sobolev-to-Lipschitz property}, i.e., 
\begin{enumerate}[label=(P4),ref=P4,itemsep=1ex]

\item\label{item:ass_stl} 
if $f\in\sob^{1,2}(X)$ is such that $|\slope f|_w\le1$, then $f$ admits a continuous representative~$\widetilde f$  such that $\widetilde f\in\lip(X)$ with $\lip(\widetilde f)\le 1$,

\end{enumerate}
to the assumptions in \cref{subsec:ass}, and by combining~\cite{G22}*{Cor.~3.21} with a plain approximation argument, we easily infer that $(\heat_t)_{t\ge0}$ satisfies~\eqref{eq:c-Feller} with
\begin{equation}
\label{eq:wbe_sfeller}
\cc(t)
\le
\left(2\int_0^t\kappa^{-2}(s)\di s\right)^{-2}
\quad
\text{for all}\ 
t\ge0.
\end{equation} 

According to~\cite{G22}*{Cor.~3.7}, if~\eqref{eq:wbe} is met by some Borel function $\kappa$ such that 
\begin{equation}
\label{eq:wbe_limsup0}
\limsup_{t\to0^+}\kappa(t)<\infty,
\end{equation} 
then the optimal function $\kappa_\star$ satisfying~\eqref{eq:wbe} is such that
\begin{equation}
\label{eq:wbe_exp_bound}
\kappa_\star(t)\le Me^{-Kt}
\quad
\text{for all}\ t\ge0
\end{equation}
for some $M\ge1$ and $K\in\R$.
Therefore, assuming~\eqref{eq:wbe_limsup0}, we can plug~\eqref{eq:wbe_exp_bound} in~\eqref{eq:wbe_sfeller} and get~\eqref{eq:c-Feller} with $\cc(t)\le M\sqrt{j_K(t)}$ for all $t>0$, where 
\begin{equation*}
j_K(t)
=
\begin{cases}
\dfrac{K}{e^{2Kt}-1} & \text{for}\ K\ne0
\\[3ex]
\dfrac{1}{2t} & \text{for}\ K=0.
\end{cases}
\end{equation*}
In particular, the bound~\eqref{eq:controls}, as well as the stronger~\eqref{eq:strong_control}, are both satisfied with $b=\frac12$.

\subsection{Synthetic constant lower curvature bounds}

The class of $\RCD(K,\infty)$ spaces, $K\in\R$, meet \eqref{eq:wbe} with $\kappa(t)=e^{-Kt}$ for $t\ge0$ and thus are a particular instance of spaces satisfying a weak Bakry--\'Emery condition~\cite{AGS14}.
The validity of~\eqref{eq:c-Feller} in $\RCD(K,\infty)$ spaces with $\cc(t)=\sqrt{j_K(t)}$ for $t>0$ has been established in~\cite{AGS14}*{Th.~6.5}, and subsequently improved to $\cc(t)=\sqrt{\frac{2}{\pi}\,j_K(t)}$ for $t>0$ in~\cite{DeMo}*{Prop.~3.1}, which is sharp for $t\to 0^+$ as a consequence of the results in~\cite{DeMoSe}.
Our results hence encode the ones in~\cites{DeMo,DF22}. 

\subsection{Variable lower curvature bounds}

In~\cite{BG21}, the authors studies the consequences of the variable lower bound $\mathrm{Ric}_g(x)(v,v)\ge k(x)\,|v|^2$, for every $x\in M$ and $v\in T_xM$, on a smooth, geodesically complete, non-compact and connected Riemannian manifold $(M,g)$ without boundary, where $k\colon M\to[0,\infty)$ is a continuous function.
Under a suitable integrability assumption on the negative part~$k^-$ of the function~$k$ (precisely, see~\cite{BG21}*{Eq.~(1.1)}, as well as the definition of the \emph{Kato class} in~\cite{BG21}*{Def.~1.2}), in~\cite{BG21}*{Th.~1.1(iii)} they establish~\eqref{eq:c-Feller} with $\cc(t)=\sqrt{8}\,t^{-1/2}\,\alpha_{k^-}(t)$ for $t>0$, where $\alpha_{k^-}\ge1$ is a function depending on the integrability condition imposed on $k^-$.

\subsection{Sub-Riemannian manifolds}

Sub-Riemannian manifolds (endowed with a smooth volume form) are infinitesimally Hilbertian spaces that do not satisfy the $\CD(K,\infty)$ property for any $K\in\R$~\cite{RS23}.
Nevertheless, numerous sub-Riemannian manifolds do enjoy~\eqref{eq:c-Feller}:
non-abelian \emph{Carnot groups}~\cites{Mel11,GT19} and the Grushin plane~\cites{W14}, both with $\cc(t)=C\sqrt{j_0(t)}$ for $t>0$ with $C>1$, and
the \emph{$\mathbb{SU}(2)$ group}~\cite{BaBo09} with $\cc(t)=C\sqrt{j_K(t)}$ for $t>0$ with $C>1$ and $K>0$.
Noteworthy, \eqref{eq:c-Feller} has also been proved in~\cite{BaBo12}*{Cor.~3.3} and~\cite{DeSu22}*{Cor.~4.11} under suitable generalized $\CD$-type conditions~\cites{BG17,Mil21}. 

\subsection{Other settings}\label{sec:other}

We believe that our approach may be naturally adapted to several other frameworks.
Here we only mention that $\leb^\infty$-to-$\lip$ contraction inequalities analogous to~\eqref{eq:c-Feller} have been established relatively to \emph{metric graphs}~\cite{BK19}, 
\emph{diamond fractals}~\cite{Alon21},
the \emph{rearranged stochastic heat equation}~\cite{DeHa22}, and the
\emph{Dyson Brownian motion}~\cite{Su23}.
  
\begin{remark}
The extension of our analysis to \emph{extended} metric-measure spaces requires some caution.
We only mention that, in this more general framework, there exist bounded Lipschitz functions (with respect to the \emph{extended} distance) which are not even measurable, see~\cite{DeSu21}*{Exam.~3.4}.
We thank Lorenzo Dello Schiavo for pointing this issue to us.
\end{remark}

\begin{remark}
All results in our paper can be expressed in the language of energy-measure spaces instead of that of metric-measure spaces. However, we note that the two approaches are equivalent under rather general assumptions, see~\cite{AGS15}*{Def.~3.14 and Th.~3.14}.
\end{remark}

\section*{Declarations}

\subsection*{Ethical approval}
Not applicable.

\subsection*{Funding}
The first-named author has received funding from INdAM under the INdAM--GNAMPA 2025 Project \textit{Proprietà qualitative e regolarizzanti di equazioni ellittiche e paraboliche} (grant agreement No.\ CUP\_E53\-240\-019\-500\-01) and the INdAM--GNAMPA 2024 Project \textit{Mancanza di regolarità e spazi non lisci: studio di autofunzioni e autovalori} (grant agreement No.\ CUP\_E53\-C23\-001\-670\-001).
The second-named author has received funding from INdAM under the INdAM--GNAMPA Project 2025 \textit{Metodi variazionali per problemi dipendenti da operatori frazionari isotropi e anisotropi} (grant agreement No.\ CUP\_E53\-240\-019\-500\-01), the INdAM--GNAMPA Project 2024 \textit{Ottimizzazione e disuguaglianze funzionali per problemi geometrico-spettrali locali e non-locali} (grant agreement No.\ CUP\_E53\-C23\-001\-670\-001) and the INdAM--GNAMPA 2023 Project \textit{Problemi variazionali per funzionali e operatori non-locali} (grant agreement No.\ CUP\_E53\-C22\-001\-930\-001), from the European Research Council (ERC) under the European Union’s Horizon 2020 research and innovation program (grant agreement No.~945655), and from  the European Union -- NextGenerationEU and the University of Padua under the 2023 STARS@UNIPD  Starting Grant Project \textit{New Directions in Fractional Calculus -- NewFrac} (grant agreement No.\ CUP\_C95\-F21\-009\-990\-001).

\subsection*{Availability of data and materials}
Not applicable.


\end{document}